\documentclass[10pt]{article}

\usepackage[final]{graphicx}
\usepackage{amsopn}
\usepackage{amstext}
\usepackage{amsfonts}
\usepackage{amssymb}
\usepackage{amscd}
\usepackage{amsthm}
\usepackage{amsfonts}
\usepackage[centertags,reqno]{amsmath}
\usepackage{latexsym}
\usepackage{hyperref}

\numberwithin{equation}{section}
\newcommand{\R}{\ensuremath{\mathbb{R}}}% real numbers
\newcommand{\Sym}{\ensuremath{\mathbb{SR}}}% symmetric matrices
\newcommand{\slab}{\ensuremath{B_{\alpha\beta}}}
\newcommand{\yayik}{\ensuremath{Q_{\alpha\beta}}}

\DeclareMathOperator{\Aut}{Aut}
\DeclareMathOperator{\cg}{cg}
\DeclareMathOperator{\co}{conv}
\DeclareMathOperator{\diag}{diag}

\DeclareMathOperator{\ext}{ext}

\DeclareMathOperator{\ce}{CE}
\DeclareMathOperator{\ie}{IE}

\DeclareMathOperator{\st}{s.t.\hspace{0.1cm}}
\DeclareMathOperator{\tr}{tr}

\DeclareMathOperator{\vol}{vol}

\DeclareGraphicsExtensions{.pdf}

\theoremstyle{plain}
\newtheorem{theorem}{Theorem}[section]
\newtheorem{lemma}[theorem]{Lemma}
\newtheorem{corollary}[theorem]{Corollary}

\theoremstyle{definition}
\newtheorem{definition}[theorem]{Definition}
\newtheorem{remark}[theorem]{Remark}

\title{The extremal volume ellipsoids of convex bodies, their symmetry properties,
        and their determination in some special cases}
\author{Osman G{\"u}ler and F\.{i}l\.{i}z G{\"u}rtuna\footnote{Department of Mathematics and Statistics,
University of Maryland, Baltimore County, Baltimore, Maryland 21250, USA.
E-mail: \emph{\{guler,gurtuna1\}@math.umbc.edu}. Research partially
supported by the National Science Foundation under grant DMS--0411955 }}

\date{September 2007}

\begin{document}

\maketitle

\bibliographystyle{acm}

\begin{abstract}
A convex body $K$ in $\R^n$ has associated with it a unique circumscribed ellipsoid $\ce(K)$ 
with minimum volume, and a unique inscribed ellipsoid $\ie(K)$ with maximum volume. We first 
give a unified, modern exposition of the basic theory of these extremal ellipsoids using the 
semi--infinite programming approach pioneered by Fritz John in his seminal 1948 paper.  We then 
investigate the automorphism groups of convex bodies and their extremal ellipsoids.  We show 
that if the automorphism group of a convex body $K$ is large enough, then it is possible to 
determine the extremal ellipsoids $\ce(K)$ and $\ie(K)$ exactly, using either semi--infinite 
programming or nonlinear programming.  As examples, we compute the extremal ellipsoids when the 
convex body $K$ is the part of a given ellipsoid between two parallel hyperplanes, and when $K$ 
is a truncated second order cone or an ellipsoidal cylinder.
\end{abstract}

\vspace{2cm}

\noindent{\bf Abbreviated title:} Extremal ellipsoids  \\
\noindent{\bf AMS(MOS) subject classifications:} primary: 90C34, 46B20, 90C30, 90C46, 65K10;
secondary: 52A38, 52A20, 52A21, 22C05, 54H15.

%   90C34:  Semi-infinite programming,
%   46B20:  convex and discrete geometry 
%   90C30:  Nonlinear programming,
%   90C46:  Optimality conditions, duality [See also 49N15]
%   65K10:  Optimization and variational techniques [See also 49Mxx, 93B40]

%   52A38:  Length, area, volume [See also 26B15, 28A75, 49Q20]
%   52A20:  Convex sets in $n$ dimensions (including convex hypersurfaces)
%   52A21:  finite dimensional Banach spaces
%   22C05:  compact groups 
%   54H15:  transformation groups         

%

\setcounter{page}{0}
\newpage

%*************************************************************************************************

\section{Introduction}
\label{sec:introduction}

A \emph{convex body} in $\R^n$ is a compact convex set with nonempty interior.
Let $K$ be a convex body in $\R^n$. Among the ellipsoids circumscribing $K$, there exists a unique one with
minimum volume and similarly, among the ellipsoids inscribed in $K$, there exists a unique one of maximum volume.
These are called the \emph{minimal circumscribed ellipsoid} and \emph{maximal inscribed ellipsoid} of $K$, and we
denote them by $\ce(K)$ and $\ie(K)$, respectively. To our knowledge, Behrend~\cite{Behrend38} is the first person to
investigate these problems, and proves the existence and uniqueness of the two ellipsoids in the plane, that is
when $n=2$. In the general case, the existence of either ellipsoid is easy to prove using compactness. The ellipsoid
$\ce(K)$ is often referred to as \emph{L{\"o}wner ellipsoid} since L{\"o}wner has used its uniqueness
in his lectures, see \cite{DanzerGrunbaumKlee63}.  The uniqueness of $\ce(K)$ also follows from the famous paper of John~\cite{John48}
although he does not state it explicitly.  Subsequently, Danzer, Laugwitz, and Lenz~\cite{DanzerLaugwitzLenz57},
and Zaguskin~\cite{Zaguskin58} prove the uniqueness of both ellipsoids.  The inscribed ellipsoid problem
$\ie(K)$ is often called \emph{John ellipsoid}, and sometimes \emph{L{\"o}wner--John} ellipsoid, especially in
Banach space geometry literature. This designation may seem inappropriate, since John does not consider the problem
$\ie(K)$ in \cite{John48}. However, in Banach space geometry literature, one is mainly interested in symmetric
convex bodies ($K=-K$), and for this class of convex bodies the inscribed and circumscribed ellipsoids are
related by polarity, so that results about one ellipsoid may be translated into a similar statement about the
other.

The two extremal problems $\ce(K)$ and $\ie(K)$ are important in several fields including optimization, Banach space
geometry, and statistics. They also have applications in differential geometry \cite{Laugwitz65}, Lie group theory
\cite{DanzerLaugwitzLenz57}, and symplectic geometry, among others.
The ellipsoid algorithm of Khachian~\cite{Khachian79} for linear programming sparked general interest of optimizers
in the circumscribed ellipsoid problem.  At the $k$th step of this algorithm, one has an ellipsoid $E^{(k)}$ and is interested
in finding $E^{(k+1)}=\ce(K)$ where $K$ is the intersection of $E^{(k)}$ with a half plane whose bounding hyperplane
passes through the center of $E^{(k)}$.  The ellipsoid $E^{(k+1)}$ can be computed explicitly, and the ratio of the
volumes $\vol(E^{(k+1)})/\vol(E^{(k)})$ determines the rate of convergence of the algorithm and gives its
polymomial--time complexity.  From this perspective, ellipsoids covering the intersection of an ellipsoid with two
halfspaces whose bounding hyperplanes are parallel have been studied as well, since the resulting ellipsoid
algorithms are likely to have faster convergence.  The papers {of K{\"o}nig and Pallaschke~\cite{KonigPallaschke81}
and Todd~\cite{Todd82} compute the circumscribing ellipsoid explicitly. It has been discovered that the circumscribed 
ellipsoid problem is dually related to the \emph{optimal design} problem in statistics. Consequently, there has been 
wide interest in the problem $\ce(K)$ and related problems in this community as well. Inscribed ellipsoid problems also 
arise in optimization. For example, the \emph{inscribed ellipsoid method} of 
Tarasov, Khachian, and Erlikh~\cite{TarasovKhachiyanErlikh88} is a polynomial--time algorithm for solving general 
convex programming problems.  In this method, one needs to compute numerically an approximation to the inscribed 
ellipsoid of a polytope which is described by a set of linear inequalities.  

To meet the demand from different fields such as optimization, computer science, engineering, and statistics,
there have been many algorithms proposed to numerically compute the two extremal ellipsoids $\ce(K)$ and $\ie(K)$.
Recently, there has been a surge of research activity in this subject. We do not discuss algorithms in this paper,
but the interested reader can find more information on this topic and the relevant references in the papers 
\cite{ZhangGao03}, \cite{SunFreund04}, and \cite{Yildirim06}.

This paper has three, related goals.  In \S\ref{sec:MVCE} and \S\ref{sec:MVIE}, we give semi--infinite programming
formulations of the problems $\ce(K)$ and $\ie(K)$, respectively, and then describe their fundamental properties using
the resulting optimality conditions. There are essentially no new results in these sections. We include them in order
to fill a gap in the optimization literature, and we also use some of the results in these sections later on in the paper. 
There exists a sizable literature in Banach space theory dealing with the existence, uniqueness, and basic properties 
of the extremal ellipsoids $\ce(K)$ and $\ie(K)$.  They use (necessarily) optimality considerations to arrive at the 
results they are interested in, but this is done in an informal manner. In fact, many of the papers in this literature 
use the reformulation of the ellipsoid problems given in the interesting paper of Lewis~\cite{Lewis79}; see also 
\cite{Pisier89}, \cite{TomczakJaegermann89}. There have been 
some exceptions recently.  See for example the interesting papers Gordon et al.~\cite{GordonLitvakMeyerPajor04}
and Klartag~\cite{Klartag04}.  To our knowledge, the only papers which systematically use optimization techniques
to prove results about the extremal ellipsoids are the original paper of John~\cite{John48} and the papers of
Juhnke~\cite{Juhnke94,Juhnke95,Juhnke04}.  What we offer in \S~\ref{sec:MVCE} and \S\ref{sec:MVIE} is a careful, 
unified, and modern synthesis of the basic results on the ellipsoids $\ce(K)$ and $\ie(K)$ in the mainstream 
optimization literature. We hope that it serves as a useful introduction to the subject in the optimization community.

We devote \S\ref{sec:automorphism-group-of-convex-bodies} to the symmetry properties of convex bodies and the related
symmetry properties of the corresponding ellipsoids $\ce(K)$ and $\ie(K)$. One way to formalize the symmetry properties
of a convex body $K$ is to consider its \emph{(affine) automorphism group} $\Aut(K)$. It will be seen that the uniqueness
of the two ellipsoids imply that the ellipsoids \emph{inherit} the symmetry properties of the underlying convex body $K$. 
That is, $\Aut(K)$ is contained in the automorphism group of the two extremal ellipsoids.  One consequence of this is that if $K$ is
``symmetric'' enough, then either it is possible to anaytically compute the extremal ellipsoids exactly, or else it is possible
to reduce the complexity of their numerical computation. In this section, we demonstrate the former possibility for a class of
convex bodies $K$ whose automorphism groups act \emph{transitively} on $\ext(K)$, the extreme points of $K$.
Davies~\cite{Davies74} shows that for this class of convex bodies, the center of gravity, the center of $\ce(K)$, and the center
of $\ie(K)$ all coincide, and this center can be obtained explicitly as a Haar integral over the automorphism group $\Aut(K)$.
We show that the matrix $X$ in the circumscribed ellipsoid $\ce(K)$ can also be obtained in a similar manner. We remark that there
has been a continuous interest on symmetric convex bodies since antiquity, and the class of \emph{regular polytopes}
\cite{Coxeter63}, including the Platonic solids, is a subclass of symmetric convex bodies in the sense of Davies.
Because of space considerations, we are not able to pursue the study of the automorphism groups of convex bodies
in greater depth in this paper. We plan to explore this subject in future papers.

In the rest of the paper, starting with \S\ref{sec:invariance-properties-of-a-slab}, we exploit automorphism
groups to analytically compute the extremal ellipsoids of two classes of convex bodies. The first class
consists of convex bodies which are the intersections of a given ellipsoid with two halfspaces whose bounding
hyperplanes are parallel, and have been mentioned above. We call such convex bodies \emph{slabs}.
The second class consists of convex bodies obtained by taking the convex hull of the intersection of the same
parallel hyperplanes with the ellipsoid. We note that a convex body in this class is either a truncated second
order cone or an ellipsoidal cylinder, depending on the location of the bounding hyperplanes with respect to
the center of the ellipsoid.

In \S\ref{sec:invariance-properties-of-a-slab}, we compute the autormorphism group of a slab $K$ and use it
to determine the form of the center and matrix of its ellipsoid $\ce(K)$.  Although the automorphism group $\Aut(K)$
is not large enough to compute the ellipsoid $\ce(K)$ exactly, it is large enough to reduce its determination to
computing just three parameters (instead of $n(n+3)/2$ in the general case), one to determine its center and two
to determine its matrix.  

In \S\ref{sec:MVCE-of-a-slab}, we formulate the $\ce(K)$ problem for a slab as a
semi--infinite programming problem, and obtain its solution by computing the three parameters of $\ce(K)$ directly
from the Fritz John optimality conditions for the semi--infinite program. As we mentioned already,
K{\"o}nig and Pallaschke~\cite{KonigPallaschke81} and Todd~\cite{Todd82} solve this exact problem.
K{\"o}nig and Pallaschke's approach is similar to ours: they use the uniqueness and invariance properties of the
ellipsoid $\ce(K)$. However, their solution is not complete since they only consider the cases when the slab
does not contain the center of the given ellipsoid.  Todd gives a complete proof covering all cases.  His proof is based
on guessing the optimal ellipsoid and then proving its minimality by using some bounds on the volume of a
covering ellipsoid. In \S\ref{sec:MVCE-of-a-slab}, we also formulate the ellipsoid problem $\ce(K)$ as a nonlinear
programming problem, and give a second, independent solution for it.

For interesting applications of the ellipsoid $\ce(K)$ of a slab, see the papers
\cite{MonteiroONealNemirovski04} and \cite{BarnesMoretti05}.

In \S\ref{sec:MVIE-of-a-slab}, we formulate the $\ie(K)$ problem for a slab as a semi--infinite programming
problem, and obtain its solution directly from the resulting Fritz John optimality conditions.
We also formulate the same problem as a nonlinear programming problem, but do not provide its solution in order
to keep the the length of the paper within reasonable bounds.

Finally, in \S\ref{sec:MVCE-of-convex-combination}, we formulate the $\ce(K)$ problem for a convex body from the second 
class of convex bodies mentioned above as a semi--infinite
programming problem, and obtain its solution directly from Fritz John optimality conditions. The form of the optimal
ellipsoid $\ce(K)$ turns out to be very similar to the form of the corresponding ellipsoid for a slab. We do not
solve the inscribed ellipsoid problem for the second class of convex bodies for space considerations.

We remark that the ideas and techniques used in this paper for determining the extremal ellipsoids for specific
classes of convex bodies can be generalized to other classes of convex bodies as long as these bodies have large enough
automorphism groups. It is reasonable to expect that automorphism groups can also be used advantageously in
numerical determination of extremal ellipsoids.

Our notation is fairly standard.  We denote the set of symmetric $n\times n$ matrices by $\Sym^{n\times n}$.  
In $\R^n$, we use the bracket notation for inner products, thus $\langle{u,v}\rangle=u^Tv$. 
In the vector space $\R^{n\times n}$ of $n\times n$ matrices (and hence in $\Sym^{n\times n}$), we use the trace 
inner product  
\[ \langle{X,Y}\rangle=\tr(XY^T). \]
If both inner products are used within the same equation, then the meaning of each inner product should be clear 
from the context. We define and use additional inner products in this paper, especially in 
\S\ref{sec:automorphism-group-of-convex-bodies}. The sets $\partial X$ and $\co(X)$, and $\ext(X)$ denote the boundary 
and the convex hull of a set $X$ in $\R^n$, respectively, and $\ext(K)$ is the set of extreme points of a convex 
set $K$ in $\R^n$.

%*************************************************************************************************

\section{The minimum volume circumscribed ellipsoid problem}
\label{sec:MVCE}

We recall that the circumscribed ellipsoid problem is the problem of finding a minimum volume
ellipsoid circumscribing a convex body $K$ in $\R^n$. This is the main problem treated in
Fritz John~\cite{John48}. In this paper, John shows that such an ellipsoid exists and is
unique; we denote it by $\ce(K)$. John introduces semi--infinite programming and develops
his optimality conditions to prove the following deep result about the ellipsoid $\ce(K)$: the
ellipsoid with the same center as $\ce(K)$ but shrunk by a factor $n$ is contained
in $K$, and if $K$ is symmetric ($K=-K$), then $\ce(K)$ needs to be shrunk by a
smaller factor $\sqrt{n}$ to be contained in $K$.  This fact is very important in the
geometric theory of Banach spaces. In that theory,  a symmetric convex body $K$ is
the unit ball of a Banach space, and if $K$ is an ellipsoid, then the Banach space
is a Hilbert space.  Consequently, the shrinkage factor indicates how close the
Banach space is to being a Hilbert space.  In this context, it is not important
to compute the exact ellipsoid $\ce(K)$. However, in some convex programming algorithms, 
including the ellipsoid method and its variants, the exact or nearly exact ellipsoid
$\ce(K)$ needs to be computed. If $K$ is sufficiently simple, $\ce(K)$ can be computed
analytically.  In more general cases, numerical algorithms have been developed to
approximately compute $\ce(K)$.

In this section, we deal with the $\ce(K)$ problem more or less following John's
approach.  However, in the interest of brevity and clarity, we use more modern
notation and give new and simpler proofs for some of the technical results.

An ellipsoid $E$ in $\R^n$ is an affine image of the unit ball
$\mathit{B_n}:=\{u\in\R^n : ||u||\leq 1\}$, that is, 
\begin{equation}
\label{eq:ellipsoid-def-1}
  E =  c+A(\mathit{B}_n) =  \{c+Au : u\in\R^n, ||u||=1\} \subset\R^m,
\end{equation}
where $A\in\R^{m\times n}$ is any $m\times n$ matrix. Here $c$ is the center of $E$
and the volume of $E$ is given by $\vol(E)=\det(A)\vol(\mathit{B}_n)$.
We are interested in the case where $E$ is a solid body ($E$ has a non--empty interior),
hence we assume that $A$ is a non--singular $n\times n$ matrix.  Let $A$ have
the singular value decomposition $A=V_1\Sigma V_2$ where $V_1$, $V_2$ are orthogonal
$n\times n$ matrices and $\Sigma$ is a diagonal matrix with positive elements.
Then we have the polar decomposition of $A$, that is, $A=SO$ where
$S=V_1AV_1^T\in\Sym^{n\times n}$ is positive definite and $O=V_1V_2$ is an orthogonal matrix.
Consequently, $E=c+SO(\mathit{B}_n)=c+S(\mathit{B}_n)$, that is, the matrix
$A$ in the definition of the ellipsoid
$E$ in \eqref{eq:ellipsoid-def-1} can be taken to be
symmetric and positive definite, an assumption we make from here on.
By making the change of variables $x:=c+Au$, that is, $u=A^{-1}(x-c)$,
and defining $X:=A^{-2}$, the ellipsoid $E$ in \eqref{eq:ellipsoid-def-1},
$E=\{x\in\R^n : ||A^{-1}(x-c)||^2\leq\,1\}$ can also be
written in the form
\begin{equation}
\label{eq:E(X,c)}
     E=E(X,c)
:=
    \{x\in\R^n : \langle X(x-c),\,x-c\rangle\,\leq\,1\}.
\end{equation}
Note that we have
\begin{equation}
\label{eq:vol(E(X,c))}
 \vol(E) = \det(X)^{-1/2}\omega_n,
\end{equation}
where $\omega_n=\vol(\mathit{B}_n)$.

Consequently, we can set up the circumscribed ellipsoid problem
as a semi--infinite program
\begin{equation}
\label{eq:MVCE-SIP}
\begin{split}
    \min\quad  &-\log\det X\qquad\qquad\qquad\qquad \\
    \st\quad   &\langle{X(y-c),y-c}\rangle\leq1,\quad \forall\,y\in K,
\end{split}
\end{equation}
in which the decision variables are $X\in\Sym^{n\times n}$ and $c\in\R^n$.

There exists an ellipsoid of minimum volume circumscribing the convex body $K$. It suffices to prove that 
the set of feasible $(X,c)$ in problem \eqref{eq:MVCE-SIP} is compact. Let $K$ contain a ball of radius $r>0$, 
and let $E=E(X,c)$ be an ellipsoid covering $K$. Note that the ball still lies in $E$ if we shift 
its center to center $c$ of the ellipsoid $E$.  Thus, every vector $u$ in $\R^n$, $||u||=1$ must satisfy the 
inequality $\langle{Xu,u}\rangle\leq1/r^2$, that is, the eigenvalues of $X$ are at most $1/r^2$. It follows from 
the spectral decomposition of symmetric matrices that the set of feasible $X$ is compact. If the norm of the 
center $c$ of the ellipsoid $E$ circumscribing $K$ goes to infinity, then the volume of the ellipsoid goes to 
infinity as well.  This proves that the set of feasible $(X,c)$ is compact.

The following basic theorem of Fritz John~\cite{John48} is one of our main tools
in this paper.  The book \cite{Guler08} develops optimality conditions for
semi--infinite programming including this result and treats several problems
from analysis and geometry using semi--infinite programming techniques.

\begin{theorem}
{\bf (\emph{Fritz John})}
\label{Theorem:FJ-Theorem-For-SIP}
Consider the optimization problem
\begin{equation}
\label{eq:SIP}
\begin{split}
   \min\quad  &f(x)  \\
   \st\quad   &g(x,y)\leq0,\quad \forall \,y\in Y,
\end{split}
\end{equation}
where $f(x)$ is a continuously differentiable function defined on an open
set $X\subseteq\R^n$, and $g(x,y)$ and $\nabla_xg(x,y)$ are continuous
functions defined on $X\times Y$ where $Y$ is a compact set in some
topological space. If $x$ is a local minimizer of \eqref{eq:SIP},
then there exist at most $n$ active constraints $\{g(x,y_i)\}_1^k$
($g(x,y_i)=0$) and a non--trivial, non--negative multiplier vector
$0\neq(\lambda_0,\lambda_1,\ldots,\lambda_k)\geq0$ such that
\[ \lambda_0\nabla f(x)+\sum_{i=1}^k\lambda_i\nabla_x g(x,y_i)
= 0. \]
\end{theorem}

We now derive the optimality conditions for \eqref{eq:MVCE-SIP}.

\begin{theorem}
\label{Theorem:FJ-conditions-for-MVCE-SIP}
Let $K$ be a convex body in $\R^n$. There exists an ellipsoid of minimum volume
circumscribing $K$. If $E(X,c)$ is such an ellipsoid, then
there exists a multiplier vector $\lambda=(\lambda_1,\ldots,\lambda_k)>0$,
$0\leq k\leq n(n+3)/2$, and points $\{u_i\}_1^k$ in $K$ such that
\begin{equation}
\label{eq:FJ-for-MVCE-general}
\begin{split}
       X^{-1}
&=     \sum_{i=1}^k\lambda_i(u_i-c)(u_i-c)^T,\\
       0
&=     \sum_{i=1}^k\lambda_i(u_i-c),\\
       u_i
&\in   \partial K\cap\partial E(X,c),
        \quad i=1,\ldots,k,\\
       K
&\subseteq
       E(X,c).
\end{split}
\end{equation}
\end{theorem}

We call the points $\{u_i\}_1^k$ in $\partial K\cap\partial E$
\emph{contact points} of $K$ and $E$.

\begin{proof}
The existence of a minimum volume circumscribed ellipsoid is already proved above.  Let $E(X,c)$ be
such an ellipsoid. The constraints in \eqref{eq:MVCE-SIP} are indexed by $y\in K$, a
compact set, and  Theorem~\ref{Theorem:FJ-Theorem-For-SIP} implies that there
exists a non--zero multiplier vector
$(\lambda_0,\lambda_1,\ldots,\lambda_k)\geq0$, where $k\leq n(n+1)/2+n=n(n+3)/2$,
$\lambda_i>0$ for $i>0$, and points $\{u_i\}_1^k$ in $K$ such that the Lagrangian function
\begin{equation*}
\begin{split}
      L(X,c,\lambda)
&:=   -\lambda_0\log\det X
        +\sum_{i=1}^k\lambda_i\langle{X(u_i-c),u_i-c}\rangle\\
&=    -\lambda_0\log\det X
        +\langle{X,\sum_{i=1}^k\lambda_i(u_i-c)(u_i-c)^T}\rangle,
\end{split}
\end{equation*}
where the inner product on the last line is the trace inner product on $S^n$,
satisfies the optimality conditions
\begin{equation*}
\begin{split}
      0
&=    \nabla_cL(X,c,\lambda)
=     X\sum_{i=1}^k\lambda_i(u_i-c),\\
      0
&=    \nabla_XL(X,c,\lambda)
=     -\lambda_0X^{-1}
       +\sum_{i=1}^k\lambda_i(u_i-c)(u_i-c)^T,
\end{split}
\end{equation*}
where we used the well--known fact that $\nabla_X\log\det X=X^{-1}$. If $\lambda_0>0$, then
$0=\tr(\sum_{i=1}^k\lambda_i(u_i-c)(u_i-c)^T)=\sum_{i=1}^k\lambda_i||u_i-c||^2$.
This implies that $\lambda_i=0$ for all $i$, contradicting $\lambda\neq0$.
We let $\lambda_0=1$ without loss of generality,  and arrive at the Fritz John
conditions \eqref{eq:FJ-for-MVCE-general}.
\end{proof}

\begin{remark}
\label{Remark:contact-points-of-MVCE-1}
The contact points have applications in several fields, in optimal designs, and
in estimating the size of almost orthogonal submatrices of orthogonal matrices
\cite{Rudelson97}, for example. Gruber~\cite{Gruber88} shows that ``most'' convex
bodies $K$ have the maximum number $n(n+3)/2$ of contact points. See also
\cite{Rudelson97} for a simpler proof. Similar results also hold for the maximum
volume inscribed ellipsoids discussed in \S\ref{sec:MVIE},
see \cite{Gruber88}.  However,  Rudelson~\cite{Rudelson97} shows that
for every $\varepsilon>0$ and every convex body $K$, there exists a nearby convex body
$L$ whose distance (Banach--Mazur distance) to $K$ is less than $1+\varepsilon$ and
which  at most $k\leq C(\varepsilon)\cdot n\log^3n$ contact points.  This has obvious
implications for numerical algorithms that try to compute approximate covering ellipsoids.
\end{remark}

\begin{remark}
\label{Remark:contact-points-of-MVCE-2}
Let $S=\ext(K)$, the set of extreme points of $K$.  We have $K=\co(S)$, the convex hull of $S$, by a Theorem of
Minkowski, see Rockafellar~\cite{Rockafellar97}, Corollary 18.5.1. Note that $S$ and $K$
have the same extremal covering ellipsoid, and applying Theorem~\ref{Theorem:FJ-conditions-for-MVCE-SIP}
to $S$ instead of $K$ shows that we can choose $u_i\in S=\ext(K)$, $i=1,\ldots,k$.

An independent proof of the above fact runs as follows: let $x\in\partial K\cap\partial E$
be a contact point. Noting $\partial E=\ext(E)$, we have $x\in\ext(E)$.
If $x\notin\ext(K)$, then there exist $y,z\in K$, $y\neq z$, such that $x$ lies in
the interior of the line segment $[y,z]$. However, $x\in[y,z]\subseteq E$,
contradicting the fact that $x\in\ext(E)$.
\end{remark}

The equation $\sum_{i=1}^k\lambda_i(u_i-c)=0$ in \eqref{eq:FJ-for-MVCE-general}
gives $c\in\co(\{u_i\}_1^k)$. This immediately implies

\begin{corollary}
\label{Corollary:contact-points-of-MVCE-cannot-lie-in-any-halfspace}
Let $K$ be a convex body in $\R^n$. The contact points of $\ce(K)$
are not contained in any closed halfspace whose bounding hyperplane
passes through the center of $\ce(K)$.
\end{corollary}

For most theoretical purposes, we may assume that the optimal ellipsoid is the
unit ball $E(I_n,0)$. This can be accomplished by an affine change of the coordinates,
if necessary.  This results in the more transparent optimality conditions
\begin{equation}
\label{eq:FJ-for-MVCE}
\begin{split}
       I_n
&=     \sum_{i=1}^k\lambda_iu_iu_i^T,\quad
       \sum_{i=1}^k\lambda_iu_i
=      0,\\
       u_i
&\in   \partial K\cap\partial \mathit{B}_n,\quad i=1,\ldots,k,\quad
       K
\subseteq
       \mathit{B}_n.
\end{split}
\end{equation}
Taking traces of both sides in the first equation above gives
$  n
=  \tr(I_n)=\tr(\sum_{i=1}^k\lambda_iu_iu_i^T)
=  \sum_{i=1}^k\lambda_iu_i^Tu_i=\sum_{i=1}^k\lambda_i$,
that is,
\begin{equation}
\label{eq:sum-of-lambdas-is-n}
\sum_{i=1}^k\lambda_i=n.
\end{equation}

In this section and in \S\ref{sec:MVIE},
convex duality will play an important role. If $C$ is a convex body in $\R^n$,
then the Minkowski \emph{support function} of $C$ is defined by
\[    s_C(d) :=  \max_{u\in C}\langle{d,u}\rangle. \]
It is obviously defined on $\R^n$ and is a convex function since it is a
maximum of linear functions indexed by $u$. In fact, $s_C=\delta_C^*$, where
$\delta_C$ is the indicator function of $C$ and $*$ denotes the Fenchel dual.
If $C$ and $D$ are two convex bodies, it follows from Corollary~13.1.1 in
\cite{Rockafellar97} that $C\subseteq D$ if and only if $s_C\leq s_D$.

We compute
\begin{equation}
\label{eq:s-of-E(X,c)}
\begin{split}
      s_{E(X,c)}(d)
&=    \max\left\{\langle{d,u}\rangle : \langle{X(u-c),u-c}\rangle\leq1\right\}\\
&=    \max\left\{\langle{d,c+X^{-1/2}v}\rangle : ||v||\leq1\rangle\right\}\\
&=    \langle{c,d}\rangle+\max_{||v||=1}\langle{X^{-1/2}d,v}\rangle
=     \langle{c,d}\rangle+||X^{-1/2}d||\\
&=    \langle{c,d}\rangle+\langle{X^{-1}d,d}\rangle^{1/2},
\end{split}
\end{equation}
where we have defined $v=X^{1/2}(u-c)$ or $u=c+X^{-1/2}v$.

The \emph{polar} of the set $C$ is defined by
\[    C^\circ
:=    \{d: s_C(d)\leq1\}
=     \{x:\langle{x,u}\rangle\leq1, \forall \,u\in C\}.
\]
An easy calculation shows that
\[ \left(\co(\{u_i\}_1^k)\right)^\circ
=
   \{x:\langle{x,u_i}\rangle\leq1,\ i=1,\ldots,k\}.
\]

The following is a \emph{key} result. Among other things, it implies that the optimality
conditions \eqref{eq:FJ-for-MVCE} are powerful enough to prove the uniqueness of the
minimum volume circumscribed ellipsoid as well as the uniqueness of the maximal volume
inscribed ellipsoid treated in \S\ref{sec:MVIE}.

\begin{lemma}
\label{Lemma:unit-ball-is-both-MVCE-and-IE}
Let $\{u_i\}_1^k$ ($k$ arbitrary) be a set of unit vectors in $\R^n$ satisfying the conditions
$\sum_{i=1}^k\lambda_iu_iu_i^T=I_n$ and $\sum_{i=1}^k\lambda_iu_i=0$.
Define the polytope $P=\co\left(\{u_i\}_1^k)\right)$ and its polar
$P^\circ=\{x:\langle{u_i,x}\rangle\leq1,\, i=1,\ldots,k\}$. The unit ball is both
the unique minimum volume circumscribed ellipsoid of $P$ and the unique maximum
volume inscribed ellipsoid of $P^\circ$.
\end{lemma}
\begin{proof}
Let $E(X,c)$ be any ellipsoid covering the points $\{u_i\}_1^k$. We have
$\langle{X(u_i-c),u_i-c}\rangle\leq1$ and
\begin{equation*}
\begin{split}
        n
&=      \sum_{i=1}^k\lambda_i
\geq   \sum_{i=1}^k\lambda_i\langle{X(u_i-c),u_i-c}\rangle
=       \langle{X, \sum_{i=1}^k\lambda_i(u_i-c)(u_i-c)^T}\rangle\\
&=     \langle{X, \sum_{i=1}^k\lambda_iu_iu_i^T}\rangle
         -\langle{X, \sum_{i=1}^kc\lambda_iu_i^T}\rangle
          -\langle{X, \sum_{i=1}^k\lambda_iu_ic^T}\rangle
           +\langle{X, (\sum_{i=1}^k\lambda_i)cc^T}\rangle\\
&=      \langle{X,I_n}\rangle+n\langle{X,cc^T}\rangle
=       \tr(X)+n\langle{Xc,c}\rangle\\
&\geq    n\left(\det(X)^{1/n}+\langle{Xc,c}\rangle\right).
\end{split}
\end{equation*}
Here the fourth equality follows from \eqref{eq:FJ-for-MVCE}, and the last
inequality follows from the fact that $\det(X)^{1/n}\leq\tr(X)/n$ which
is precisely the arithmetic--geometric mean inequality applied to the
eigenvalues of $X$.  Thus
\[    \det(X)^{1/n}+\langle{Xc,c}\rangle \leq  1, \]
and the equality  $\det(X)=1$ holds if and only if $c=0$,
$\langle{X(u_i-c),u_i-c}\rangle=1$ for all $i=1,\ldots,k$,
and the arithmetic--geometric mean inequality holds as an equality.
The last condition holds if and only if $X$ is a positive multiple of the
identity matrix (and then $\det(X)=1$ implies $X=I_n$).  Thus, the minimum
volume ellipsoid covering the points $\{u_i\}_1^k$ must be the unit ball.

Next, let $E(X,c)$ be any ellipsoid inscribed in $P^\circ$. It follows from \eqref{eq:vol(E(X,c))}
that $\vol(E(X,c))=\det(X^{-1})\omega_n$.
By virtue of \eqref{eq:s-of-E(X,c)}, the inclusion $E(X,c)\subseteq P^\circ$ implies
\[      s_{E(X,c)}(u_i)
=       \langle{c,u_i}\rangle+||X^{-1/2}u_i||
\leq    s_{P^\circ}(u_i)
=       \max_j\langle{u_i,u_j}\rangle
\leq    1,\quad i=1,\ldots,k.
\]
The Cauchy--Schwarz inequality gives
$\langle{X^{-1/2}u_i,u_i}\rangle\leq||X^{-1/2}u_i||\cdot||u_i||=||X^{-1/2}u_i||$.
Therefore,
\begin{equation*}
\begin{split}
       n
&=     \sum_{i=1}^k\lambda_i
\geq   \sum_{i=1}^k\lambda_i
        \left(\langle{c,u_i}\rangle+\langle{X^{-1/2}u_i,u_i}\rangle\right)\\
&=     \langle{X^{-1/2}, \sum_{i=1}^k\lambda_iu_iu_i^T}\rangle
=      \tr(X^{-1/2})
\geq   n\det(X)^{-1/2n},
\end{split}
\end{equation*}
where the last inequality follows from the arithmetic--geometric mean inequality
applied to the eigenvalues of $X^{-1/2}$.  Thus $\det(X)\geq1$, and the equality $\det(X)=1$
holds if and only if (i) $X$ is a positive multiple of the identity matrix
(and then $\det(X)=1$ implies $X=I_n$), and (ii)
$1=\langle{c,u_i}\rangle+\langle{X^{-1/2}u_i,u_i}\rangle=\langle{c,u_i}\rangle+1$,
that is, $\langle{c,u_i}\rangle=0$ for all $i=1,\ldots,k$. Then the equation
$\sum_{i=1}^n\lambda_iu_iu_i^T=I_n$ implies that
$||c||^2=\sum_{i=1}^k\lambda_i\langle{c,u_i}\rangle^2=0$. Thus, condition (ii)
holds if and only if $c=0$.  The lemma is proved.
\end{proof}

\begin{theorem}
\label{Theorem:MVCE-uniqueness-and-sufficiency}
Let $K$ be a convex body in $\R^n$. The minimum volume circumscribed ellipsoid of
$K$ is unique.  Moreover, the optimality conditions \eqref{eq:FJ-for-MVCE-general}
are \emph{necessary and sufficient} conditions for an ellipsoid $E(X,c)$ to be the
minimum volume circumscribed ellipsoid of $K$.
\end{theorem}

It is easy to see that $K$ can be replaced by $S=\ext(K)$ in this theorem.

\begin{proof}
The neccessity of the conditions \eqref{eq:FJ-for-MVCE-general} is already proved in
Theorem~\ref{Theorem:FJ-conditions-for-MVCE-SIP}. We assume, without any loss of generality,
that $E(I_n,0)=\mathit{B}_n$ satisfies the optimality conditions \eqref{eq:FJ-for-MVCE} for some 
set of multipliers $\{\lambda_i\}$. Let $E=\ce(K)$. We claim that $E=\mathit{B}_n$.  This will 
immediately imply the remaining parts of the theorem. Since $E\subseteq P$, we have 
$\vol(\mathit{B}_n)\geq\vol(E)\geq\vol(\mathit{B}_n)$ where the first inequality follows because 
$E$ has minimum volume among ellipsoids circumscribing $K$ and the second inequality follows from 
Lemma~\ref{Lemma:unit-ball-is-both-MVCE-and-IE}. The same lemma proves the claim that $E=\mathit{B}_n$.
\end{proof}

The \emph{breadth} of a convex body is the smallest distance between its two parallel
support planes, and \emph{diameter} is the distance between its two farthest points.
The following result of John~\cite{John48} shows that a convex body can be ``rounded'' by
an affine transformation.  Its proof also gives valuable information about the
locations of the contact points.  Its easy proof is due to Juhnke~\cite{Juhnke04}.

\begin{corollary}
\label{Corollary:diameter-breadth}
Let $K$ be a convex body whose optimal covering ellipsoid is the unit ball,
and let $\{u_i\}_1^k$ be its contact points. Then
\begin{equation*}
\begin{split}
           \max_{x\in K}\langle{d,x}\rangle
           \max_{x\in K}\langle{-d,x}\rangle
\,&\geq\,  \max_i\langle{d,u_i}\rangle
           \max_i\langle{-d,u_i}\rangle
\,\geq\,   \frac{1}{n},
            \quad \forall d,\ ||d||=1,\\
           \max_{x\in K}\langle{d,x}\rangle
            +\max_{x\in K}\langle{-d,x}\rangle
\,&\geq\,   \max_i\langle{d,u_i}\rangle
            +\max_i\langle{-d,u_i}\rangle
\,\geq\,   \frac{2}{\sqrt{n}},
            \quad \forall d,\ ||d||=1.
\end{split}
\end{equation*}
Consequently, any convex body can be transformed by an affine map into a body,
for which the ratio of breadth to diameter is at least $1/\sqrt{n}$.
\end{corollary}
\begin{proof}
Define $P=\co(\{u_i\}_1^k)$.
Since $s_P(d)=\max_k\langle{d,u_k}\rangle\geq\langle{d,u_i}\rangle$ for any $i$, we have
\begin{equation*}
\begin{split}
       0
&\leq  \sum_{i=1}^k\lambda_i
        (s_P(d)-\langle{d,u_i}\rangle)
         (s_P(-d)+\langle{d,u_i}\rangle)\\
&=     (\sum_{i=1}^k\lambda_i)s_P(d)s_P(-d)
        -\sum_i\lambda_i\langle{d,u_i}\rangle^2
=      ns_P(d)s_P(-d)-1
\end{split}
\end{equation*}
where the equalities follow from \eqref{eq:FJ-for-MVCE}.
This proves the first line of relations in the corollary.  The second line
of relations follow from the first one using the inequality
$(s_P(d)+s_P(-d))^2\geq4s_P(d)s_P(-d)$.  Observe that the quantity
$s_K(d)+s_K(-d)=\max_{x\in K}\langle{d,x}\rangle+\max_{x\in K}\langle{-d,x}\rangle$
is the distance between the two parallel support planes of $K$ in direction $d$.
This proves the corollary for $K$ whose optimal ellipsoid $\ce(K)=B_n$.

The rest of the corollary follows since an arbitrary convex body can be transformed
by an affine transformation into another convex body whose optimal covering
ellipsoid is the unit ball.
\end{proof}

We now give a proof of Fritz John's celebrated result mentioned at the beginning
of this section. Our proof is simpler, and uses ideas from  Ball~\cite{Ball92} and
Juhnke~\cite{Juhnke94}.

\begin{theorem}
\label{thm:CE/n-is-n-K}
Let $K$ be a convex body in $\R^n$ and $E(X,c)=\ce(K)$ be its optimal circumscribing
ellipsoid. The ellipsoid with the same center $c$ but shrunk by a factor $n$ is
contained in $K$. If $K$ is symmetric ($K=-K$), then the ellipsoid with the same
center $c$ but shrunk only by a factor $\sqrt{n}$ is contained in $K$.
\end{theorem}
\begin{proof}
Without loss of generality, we assume that $\ce(K)=E(I_n,0)=B_n$. The theorem states
that $n^{-1}B_n\subseteq K$.  Let
\[   P=\co\left(\{u_i\}_1^k\right) \]
be the convex hull of the contact points. We claim the stronger statement that
$n^{-1}B_n\subseteq P$. Since $P\subseteq K$, we will then have $n^{-1}B_n\subseteq K$.
By duality, the claim is equivalent to showing that the polar sets satisfy
$P^*\subseteq(n^{-1}B_n)^*=nB_n$.  Let $x\in P^*$. Since
$-||x||=-||x||\cdot||u_i||\leq\langle{x,u_i}\rangle\leq1$, we have
\begin{equation*}
\begin{split}
       0
&\leq  \sum_{i=1}^k\lambda_i(1-\langle{x,u_i}\rangle)(||x||+\langle{x,u_i}\rangle)\\
&=     (\sum_{i=1}^k\lambda_i)||x||-\sum_{i=1}^k\lambda_i(\langle{x,u_i}\rangle)^2
=      n||x||-||x||^2,
\end{split}
\end{equation*}
where the second equality follows from $\sum_i\lambda_i=n$ and \eqref{eq:FJ-for-MVCE}.
This implies $||x||\leq n$, and proves that $P^*\subseteq nB_n$.

If $K$ is symmetric, we define $Q=\co\left(\{\pm u_i\}_1^k\right)\subseteq K$ and
claim that $n^{-1/2}B_n\subseteq Q$, or equivalently, that
$Q^*\subseteq(n^{-1/2}B_n)^*=\sqrt{n}B_n$. It is easily shown that
$Q^*=\{x:|\langle{x,u_i}\rangle|\leq1, i=1,\ldots,k\}$.
Let $x\in Q^*$. Since $-1\leq\langle{x,u_i}\rangle\leq1$, we have
\[    0
\leq  \sum_{i=1}^k\lambda_i(1-\langle{x,u_i}\rangle)(1+\langle{x,u_i}\rangle)
=     n-||x||^2. \]
This gives $||x||\leq\sqrt{n}$ and proves the claim.
\end{proof}

\begin{remark}
\label{Remark:a-different-formulation-MVCE-problem-as-a-SIP}
The minimum volume covering ellipsoid problem may be set as a semi--infinite
program in a different way, by replacing the set inclusion $K\subseteq E(X,c)$
by the equivalent functional constraints $s_{E(X,c)}(d)\geq s_K(d)$, that is
by the constraints
\[    \langle{c,d}\rangle+\langle{Yd,d}\rangle^{1/2}
\geq  s_K(d),
       \quad \forall \,d,\,||d||=1,
\]
where we restrict $d$ to the unit sphere since support functions are homogeneous
(of degree 1), in order to get a compact indexing set.  The resulting semi--infinite
program is solved in the same way as \eqref{eq:MVCE-SIP}, and mirrors the solution to
the maximum volume inscribed ellipsoid problem given in \S\ref{sec:MVIE}.
\end{remark}

%*************************************************************************************************

\section{The maximum volume inscribed ellipsoid problem}
\label{sec:MVIE}

Recall that the \emph{inscribed ellipsoid problem} is the problem of finding a maximum
volume ellipsoid inscribed in a convex body $K$ in $\R^n$. It will be seen that this
ellipsoid is unique as well, and we denote it by $\ie(K)$. As we mentioned in the Introduction,
this ellipsoid is often referred to as the John ellipsoid or L{\"o}wner--John ellipsoid.

In this section, we again use semi--infinite programming to treat this problem.
The inscribed ellipsoid has properties similar to those of the circumscribed
ellipsoid. For example, the ellipsoid with the same center but blown up $n$ times
contains $K$, and in the case $K$ is symmetric ($K=-K$), the ellipsoid needs to be blown
up by a smaller factor $\sqrt{n}$. The ellipsoid $\ie(K)$ is very useful in the
geometric theory of Banach spaces. It is also useful in some convex programming
algorithms, such as the \emph{inscribed ellipsoid method} of
Tarasov, Erlikh, and Khachiyan~\cite{TarasovKhachiyanErlikh88}.

As a first step, using \eqref{eq:vol(E(X,c))}, we can formulate the inscribed
ellipsoid problem as a semi--infinite program
\[ \min\{\det X: E(X,c)\subseteq K\}. \]
However, this is hard to work with, due to the inconvenient form of the constraints,
$E(X,c)\subseteq K$.  We replace this inclusion by the functional constraints
\[ s_{E(X,c)}(d)\leq s_K(d),\quad \forall\,d\in B_n, \]
where we again restrict $d$ to the unit sphere since support functions are homogeneous
(of degree 1), in order to get a compact indexing set.

Defining $Y=X^{-1}$, we can therefore rewrite our semi--infinite program
in the form
\begin{equation}
\label{eq:IE-SIP}
\begin{split}
    \min\quad  &-\log\det Y\qquad\qquad\qquad\qquad \\
    \st\quad   &\langle{c,d}\rangle+\langle{Yd,d}\rangle^{1/2}\leq s_K(d),
                 \quad \forall\,d:\,\|d\|=1,
\end{split}
\end{equation}
in which the decision variables are $(Y,c)\in S^n\times\R^n$ and we have
infinitely many constraints indexed by the unit vector $||d||=1$.

Since $s_K$ is a convex function on $\R^n$, it is continuous. Therefore,
there exists a positive constant $M>0$ such that if $(Y,c)$ is a feasible
decision variable, then $|\langle{c,d}\rangle|\leq M$, and
$\langle{Yd,d}\rangle\leq M$ for all $||d||=1$.  This proves that the
set of feasible $(Y,c)$ for problem \eqref{eq:IE-SIP} is compact, and implies that
there exists a maximum volume ellipsoid inscribed in $K$.

We derive the optimality conditions for the maximum volume inscribed
ellipsoid.

\begin{theorem}
\label{Theorem:FJ-conditions-for-IE}
Let $K$ be a convex body in $\R^n$. There exists an ellipsoid of maximum
volume inscribed in $K$. If $E(X,c)$ is such an ellipsoid, then
there exists a multiplier vector $\lambda=(\lambda_1,\ldots,\lambda_k)>0$,
$0\leq k\leq n(n+3)/2$, and contact points $\{u_i\}_1^k$ such that
\begin{equation}
\label{eq:FJ-for-MVIE-general}
\begin{split}
      X^{-1}
&=    \sum_{i=1}^k\lambda_i(u_i-c)(u_i-c)^T,\\
      0
&=    \sum_{i=1}^k\lambda_i(u_i-c),\\
      u_i
&\in  \partial K\cap\partial E(X,c),\quad i=1,\ldots,k,\\
      E(X,c)
&\subseteq
      K.
\end{split}
\end{equation}
\end{theorem}
\begin{proof}
The existence of a maximum volume ellipsoid inscribed in $K$ is already
proved above. Let $E(X,c)$ be such an ellipsoid. Define $Y=X^{-1}$.
Since constraints in \eqref{eq:IE-SIP} are indexed by $||d||=1$,
Theorem~\ref{Theorem:FJ-Theorem-For-SIP} applies: there exists a non--zero
multiplier vector $(\delta_0,\delta_1,\ldots,\delta_k)\geq0$, where
$k\leq n(n+3)/2$, $\delta_i>0$ for $i>0$, and directions $\{d_i\}_1^k$, $||d_i||=1$,
satisfying the conditions
\[      \langle{c,d_i}\rangle+\langle{Yd_i,d_i}\rangle^{1/2}
=       s_K(d_i), \]
such that the Lagrangian function
\begin{equation*}
      L(Y,c,\delta)
:=   -\delta_0\log\det Y
        +2\sum_{i=1}^k\delta_i
         \left[\langle{c,d_i}\rangle+\langle{Yd_i,d_i}\rangle^{1/2}-s_K(d_i)\right]
\end{equation*}
satisfies the optimality conditions
\begin{equation*}
\begin{split}
        0
&=      \nabla_cL(Y,c,\delta)
=       \sum_{i=1}^k\delta_id_i,\\
        0
&=      \nabla_YL(Y,c,\delta)
=        -\delta_0Y^{-1}
          +\sum_{i=1}^k\frac{\delta_i}{\langle{Yd_i,d_i}\rangle^{1/2}}d_id_i^T.
\end{split}
\end{equation*}
Recalling that $||d_i||=1$ and taking the trace of the right hand side of the
last equation above gives
$     \delta_0\tr(Y^{-1})
=     \sum_{i=1}^k\delta_i\langle{Yd_i,d_i}\rangle^{-1/2}.
$
If $\delta_0=0$, then all $\delta_i=0$, which contradicts $\delta\neq0$.
Therefore, $\delta_0\neq0$, and we let $\delta_0=1$.  Define
\[     u_i
:=     c+\langle{Yd_i,d_i}\rangle^{-1/2}Yd_i,
\quad
       \lambda_i
:=     \langle{Yd_i,d_i}\rangle^{1/2}\delta_i,
        \quad i=1,\ldots,k.
\]
We have $\langle{d_i,u_i}\rangle=\langle{c,d_i}\rangle+\langle{Yd_i,d_i}\rangle^{1/2}$, so that 
\[     s_{E(X,c)}(d_i)
=      s_K(d_i)
=      \langle{d_i,u_i}\rangle, \]
which means that $u_i\in\partial K\cap\partial E(X,c)$, that is $u_i$ is a contact
point. Rewriting the above optimality conditions in terms of $\{u_i\}$ and
$\{\lambda_i\}$ and simplifying, we arrive at the conditions
\eqref{eq:FJ-for-MVIE-general}.
\end{proof}

As in the circumscribed ellipsoid case, we have
\begin{corollary}
\label{Corollary:contact-points-of-MVIE-cannot-lie-in-any-halfspace}
Let $K$ be a convex body in $\R^n$. The contact points of $\ie(K)$
are not contained in any closed halfspace whose bounding hyperplane
passes through the center of $\ie(K)$.
\end{corollary}

We can simplify these conditions by assuming that the optimal ellipsoid is
the unit ball $E(I_n,0)$. Then the Fritz John conditions become
\begin{equation}
\label{eq:FJ-for-MVIE}
\begin{split}
       I_n
&=     \sum_{i=1}^k\lambda_iu_iu_i^T,\quad
       0
=     \sum_{i=1}^k\lambda_iu_i,\\
       u_i
&\in   \partial K\cap\partial \mathit{B}_n,
        \quad i=1,\ldots,k,\quad
       \mathit{B}_n
\subseteq
       K.
\end{split}
\end{equation}
We note that the optimality conditions \eqref{eq:FJ-for-MVIE} are exactly the
\emph{same} as the corresponding optimality conditions \eqref{eq:FJ-for-MVCE}
in the circumscribed ellipsoid case, except for the feasibility constraints
$\mathit{B}_n\subseteq K$.

\begin{theorem}
\label{Theorem:IE-uniqueness-and-sufficiency}
Let $K$ be a convex body in $\R^n$. The maximal volume ellipsoid inscribed in
$K$ is unique.  Furthermore, the optimality conditions \eqref{eq:FJ-for-MVIE-general}
are \emph{necessary and sufficient} for an ellipsoid $E(X,c)$ to be the maximal volume
inscribed ellipsoid of $K$.
\end{theorem}
The proof uses Lemma~\ref{Lemma:unit-ball-is-both-MVCE-and-IE}.  It is omitted since
it is very similar to the proof of Theorem~\ref{Theorem:MVCE-uniqueness-and-sufficiency}.

We end this section by proving an analogue of Fritz John's containment results
concerning $\ce(K)$.

\begin{theorem}
\label{thm:nIE-covers-K}
Let $K$ be a convex body in $\R^n$ and let $E(X,c)=\ie(K)$ be its optimal inscribed
ellipsoid. The ellipsoid with the same center $c$ but enlarged by a factor $n$
contains $K$. If $K$ is symmetric, then the ellipsoid with the same center
$c$ but enlarged by a factor $\sqrt{n}$ contains $K$.
\end{theorem}
\begin{proof}
The proof here is similar to the proof of Theorem~\ref{thm:CE/n-is-n-K}.
Without loss of generality, we assume that $\ie(K)=E(I,0)=B_n$. The first part
of the theorem follows if we can prove the claim that
\[ K\subseteq P^*\subseteq nB_n. \]
Since $1=s_K(u_i)=\max_{x\in K}\langle{u_i,x}\rangle$, the first inclusion
holds true. If $x\in P^*$, then
$-||x||=-||x||\cdot||u_i||\leq\langle{x,u_i}\rangle\leq1$, and
\[    0
\leq  \sum_{i=1}^k\lambda_i(1-\langle{x,u_i}\rangle)(||x||+\langle{x,u_i}\rangle)
=     n||x||-||x||^2, \]
where the equality follows from $\sum_i\lambda_i=n$ and \eqref{eq:FJ-for-MVIE}.
This implies $||x||\leq n$, and proves the second inclusion in the claim.

If $K$ is symmetric, we define $Q=\co\left(\{\pm u_i\}_1^k\right)\subseteq K$
and claim that $K\subseteq Q^*\subseteq\sqrt{n}B_n$.  Since $1=s_K(\pm u_i)$,
we have $|\langle{u_i,x}\rangle|\leq1$, and the first inclusion follows.
To prove the second inclusion in the claim, let $x\in Q^*$. We have
$|\langle{x,u_i}\rangle|\leq1$, and the rest of the proof follows as in the
proof of Theorem~\ref{thm:CE/n-is-n-K}.
\end{proof}

%*************************************************************************************************

\section{Automorphism group of convex bodies}
\label{sec:automorphism-group-of-convex-bodies}

Let $K$ be a convex body in $\R^n$.  The uniqueness of the two extremal ellipsoids
$\ce(K)$ and $\ie(K)$ have important consequences regarding the invariance properties
of the two ellipsoids.  We will see in this section that the symmetry properties
of the convex body $K$ is inherited by the two ellipsoids. If $K$ is
symmetric enough, then it becomes possible to give explicit formulae for the ellipsoids
$\ce(K)$ and $\ie(K)$.  We will demonstrate this for some special convex bodies in
the remaining Sections of this paper. 

Apart from its intrinsic importance, the invariance properties of the ellipsoids $\ce(K)$
and $\ie(K)$ have important applications to Lie groups \cite{DanzerLaugwitzLenz57},
\cite{OnishchickVinberg90}, to differential geometry \cite{Laugwitz65}, and to the computation of
the extremal ellipsoids for some special polytopes and convex bodies
\cite{Blekherman04}, \cite{BarvinokBlekherman05}, among others.

We start with a
\begin{definition}
The (affine) automorphism group $\Aut(K)$ of a convex body $K$ in $\R^n$ is the set of
affine transformation $T(x)=a+Ax$ leaving $K$ invariant, that is,
\[ \Aut(K) = \{ T(x)=a+Ax: T(K)=K\}. \]
\end{definition}
Note that since $0<\vol(K)=\vol(T(K))=|\det A|\vol(K)$, we have $|\det A|=1$.

It is shown in \cite{Gruber88} that the automorphism group of most convex bodies
consists of the identity transformation alone. This is to be expected, since the
symmetry properties of a given convex body can easily be destroyed by slightly perturbing
the body. Nevertheless, the study of the symmetry properties of convex bodies is
important for many reasons.

The ellipsoids are the most symmetric convex bodies. Therefore, we first investigate
their automorphism groups and then relate them to the automorphism groups of arbitrary
convex bodies.

\begin{definition}
\label{Definition:orthogonal-X-matrices}
Let $A$ be an invertible matrix in $\Sym^{n\times n}$.
Equip $\R^n$ with the quadratic form $\langle{Au,v}\rangle$ which we write as an
inner product
\[    \langle{u,v}\rangle_A:=\langle{Au,v}\rangle. \]
We denote by $O(\R^n,A)$ the set of linear maps \emph{orthogonal}
under this inner product,
\begin{equation*}
\begin{split}
      O(\R^n,A)
&:=   \{g\in\R^{n\times n}: g^*g=gg^*=I\}\\
&=    \{g\in\R^{n\times n}: \langle{gu,gv}\rangle_A=\langle{u,v}\rangle_A,
      \,\forall\, u,v\in\R^n\} \\
&=    \{g\in\R^{n\times n}: g^TAg=A\},
\end{split}
\end{equation*}
where $g^*$ is the conjugate matrix of $g$ with respect to the inner product
$\langle{\cdot,\cdot}\rangle_A$, that is $\langle{gu,v}\rangle_A=\langle{u,g^*v}\rangle_A$
for all $u,v$ in $\R^n$. The second equality above follows as
$\langle{gu,gv}\rangle_A=\langle{u,g^*gv}\rangle_A$ and this equals $\langle{u,v}\rangle_A$
if and only if $g^*g=I$. The third  equality follows since
$     \langle{g^TAgu,v}\rangle
=     \langle{Agu,gv}\rangle
=     \langle{gu,gv}\rangle_A
=     \langle{u,v}\rangle_A
=     \langle{Au,v}\rangle.
$
If $A$ is positive definite, then $(\R^n,\langle{\cdot,\cdot}\rangle_X)$
is a Euclidean space. In particular, $O_n:=O(\R^n,I)$ is the set of orthogonal
matrices in the usual inner product on $\R^n$.
\end{definition}

\begin{lemma}
\label{Lemma:Aut(E)}
If $X$ is a symmetric, positive definite ${n\times n}$ matrix, then
\[   \Aut(E(X,0)) = O(\R^n,X). \]
In particular, $\Aut(B_n)$ is the set of $n\times n$ orthogonal matrices.
We also have
\[   \Aut(E(X,c))
=    T_cO(\R^n,X)T_{-c},
\]
where $T_c$ is the translation map $T_cx=c+x$.
Moreover, $\Aut(E(X,c))$ fixes $c$, the center of $E(X,c)$, that is,
$\theta(c)=c$ for every $\theta$ in $\Aut(E(X,c))$.
\end{lemma}
\begin{proof}
We first determine $\Aut(B_n)$.  Let $T$ be in $\Aut(B_n)$, where $T(x)=a+Ax$. Since
$T$ maps the boundary of $B_n$ onto itself, we have
$q(x):=||a+Ax||^2=\langle{A^TAx,x}\rangle+2\langle{A^Ta,x}\rangle+||a||^2=1$
for all $||x||=1$.  Then $q(x)-q(-x)=4\langle{A^Ta,x}\rangle=0$ for all
$||x||=1$, which implies that $A^Ta=0$ and since $A$ is invertible, $a=0$.
Consequently, $q(x)=\langle{A^TAx,x}\rangle=1$ for all $||x||=1$, which gives
$A^TA=I_n$, that is, $A$ is an orthogonal matrix.

Next, we determine $\Aut(E(X,0))$. We have the commutative diagram
\begin{equation*}
   \begin{CD}
   B_n              @>X^{-1/2}>> E(X,0)        @>T_c>> E(X,c)\\
   @V{U}VV          @V{\theta_0}VV   @VV{\theta}V\\
   B_n              @>X^{-1/2}>> E(X,0)        @>T_c>> E(X,c)
   \end{CD}
\end{equation*}
where $\theta\in\Aut(E(X,c))$, $U\in\Aut(B_n)=O_n$, and $\theta_0\in\Aut(E(X,0))$.
From the diagram, we have
$I=U^TU=(X^{1/2}\theta_0 X^{-1/2})^T(X^{1/2}\theta_0 X^{-1/2})
=X^{-1/2}\theta_0^TX\theta_0 X^{-1/2}$. This gives $\theta_0^TX\theta_0=X$
and proves that $\Aut(E(X,0))=O(\R^n,X)$.

Since $T_c^{-1}=T_{-c}$, we have
\[   \Aut(E(X,c)) = T_c\circ\Aut(E(X,0))\circ T_{-c}. \]
Every $\theta$ in $\Aut(E(X,c))$ has the form
$    \theta(u)
=    (T_c\circ\theta_0\circ T_{-c})(u)
=    T_c(\theta_0(-c+u))
=    (c-\theta_0c)+\theta_0u$
for some $\theta_0$ in $\Aut(E_0)$. This gives $\theta(c)=c$,
meaning that $\Aut(E(X,c))$ fixes the center of $E$.
\end{proof}

\begin{definition}
\label{Definition:invariant-ellipsoid}
Let $K$ be a convex body in $\R^n$. An ellipsoid $E=E(X,c)$ is an \emph{invariant
ellipsoid} of $K$ if $\Aut(K)\subseteq\Aut(E)$, that is, if every automorphism of $K$
is an automorphism of $E$.
\end{definition}
It immediately follows from Lemma~\ref{Lemma:Aut(E)} that $\Aut(K)$ fixes the center of
any invariant ellipsoid $E$ of $K$.

The following Theorem in Danzer et al.~\cite{DanzerLaugwitzLenz57} is a central result
regarding the symmetry properties of the extremal ellipsoids.

\begin{theorem}
\label{Theorem:invariance}
Let $K$ be a convex body in $\R^n$.  The extremal ellipsoids $\ce(K)$ and $\ie(K)$ are
invariant ellipsoids of $K$. Thus, $\Aut(K)\subseteq\Aut(\ce(K))$,
$\Aut(K)\subseteq\Aut(\ie(K))$, and $\Aut(K)$ fixes the centers of $\ce(K)$ and $\ie(K)$.
\end{theorem}
\begin{proof}
Since the arguments are similar, we only prove the statements about $\ce(K)$.
Let $g\in\Aut(K)$. Since $K\subseteq\ce(K)$ and $K=gK\subseteq g(\ce(K))$,
the ellipsoids $\ce(K)$ and $g(\ce(K))$ both cover $K$, and since
$\vol(\ce(K))=\vol(g(\ce(K)))$, they are both minimum volume circumscribed
ellipsoids of $K$. It follows from
Theorem~\ref{Theorem:MVCE-uniqueness-and-sufficiency} that $g(\ce(K))=\ce(K)$.
\end{proof}

\begin{corollary}
\label{Corollary:Aut(K)-is-compact}
The automorphism group $\Aut(K)$ of a convex body $K$ in $\R^n$ is a compact Lie group.
\end{corollary}
\begin{proof}
We have $Aut(K)\subseteq\Aut(\ce(K))$ by Theorem~\ref{Theorem:invariance}, and
Lemma~\ref{Lemma:Aut(E)} implies that $\Aut(\ce(K))$ is compact. Clearly, $\Aut(K)$
is a closed subset of the general affine linear group in $\R^n$. It follows that
$\Aut(K)$ is a compact group.  A classical theorem of von Neumann implies that it is
a Lie group.
\end{proof}

\begin{remark}
\label{Remark:invariance}
The existence of invariant (fixed) points and (symmetric positive definite)
matrices for $\Aut(K)$ have been demonstrated above using the invariance properties of
the either one of the extremal ellipsoids $\ce(K)$ and $\ie(K)$.  The same goal can be achieved
in at least two other ways, using the invariance properties of either the \emph{center of
gravity} of convex sets, or of the Haar probability measure $\mu_G$ on $G=\Aut(K)$. In a
certain sense, all three procedures are similar in that they all employ \emph{averaging}, but in
different ways.

By definition, the center of gravity of a convex body $K$ is the point
\[     \cg(K)
:=     \frac{\int_K x\,dx}{\int_Kdx}
=      \frac{\left(\int_K x_1\,dx,\ldots,\int_K x_n\,dx\right)}{\vol(K)}.
\]
One may think of $\cg(K)$ as the limit of the points
$p:=\sum_{i=1}^k(\vol(K_i)/\vol(K))x_i$ where $\{K_i\}_1^k$
is a partition of $K$ into subregions and $x_i$ in $K_i$.  Since $p\in K$, we see that $\cg(K)$ lies
in $K$.  If $T\in\Aut(K)$, then $Tp=\sum_{i=1}^k(\vol(K_i)/\vol(K))Tx_i=\sum_{i=1}^k(\vol(TK_i)/\vol(K))Tx_i$.
We see that both $\{p\}$ and $\{Tp\}$ converge to $\cg(K)$, proving that $\Aut(K)$ fixes the center of
gravity of $K$. 

\noindent
Moreover, one can prove the existence of an invariant ellipsoid for $K$, see \cite{OnishchickVinberg90},
pp. 130--135: it is easy to verify that the map $\pi:\Aut(K)\to GL(\Sym^{n\times n})$
given by the formula 
\[  \pi(g)(S)=(g\otimes g)(S):=gSg^T \]
is a \emph{representation} of the group $\Aut(K)$ on the matrix space $\Sym^{n\times n}$, 
that is, $\pi(gh)=\pi(g)\pi(h)$:
\[  \pi(gh)(S) = ghSh^Tg^T = g(\pi(h)(S)g^T = \pi(g)\pi(h)(S). \]
Consider the group $\mathcal{G}:=\pi(\Aut(K))$ and the orbit
\[ C_S=\mathcal{G}(S)=\{gSg^T: g\in\Aut(K)\} \]
where $S\in\Sym^{n\times n}$ is an arbitrary positive definite matrix. Define the convex body
$\mathcal{K}=\co(C_S)\subset\Sym^{n\times n}$. Clearly, the orbit $C_S$ is invariant under
the group $\mathcal{G}$, and thus so is the convex body $\mathcal{K}$. It follows from
the above argument that the center of gravity $Y=\cg(\mathcal{K})$ is fixed by $\mathcal{G}$,
that is, $gYg^T=Y$ for all $g$ in $\Aut(K)$.  This gives $(g^{-1})^TY^{-1}g^{-1}=S^{-1}$, or equivalently,
$g^TS^{-1}g=S^{-1}$ for all $g$ in $\Aut(K)$. This means that $X=Y^{-1}$ is invariant under $\Aut(K)$.
Then any ellipsoid $E(X,c)$, where $c\in K$ is any fixed point of $\Aut(K)$, say the center of
gravity of $K$, is an invariant ellipsoid of $K$.

Let $\mu=\mu_G$ be the unique Haar probability measure on $G=\Aut(K)$. If $f:G\to\R$ is a continuous 
function and $h,k\in G$ are arbitrary, we have 
\[ \int_G f(hg)\,d\mu(g)
=  \int_G f(g)\,d\mu(g)
=  \int_G f(gk)\,d\mu(g)
=  \int_G f(g^{-1})\,d\mu(g),
\]
where the first and second equalities express left and right invariance properties of the Haar integral, respectively.
If $x$ is any interior point of $K$, then the point $c:=\int_Ggx\,d\mu(g)$ lies in $interior(K)$ and is
fixed by $\Aut(K)$, for if $h\in\Aut(K)$, then $hc=\int_Ghgc d\,\mu(g)=\int_Ggc d\,\mu(g)=c$,
where the second equality follows from the left invariance of $\mu$.  Finally, the standard proof of
the existence of a positive definite invariant matrix $X$ found in most textbooks proceeds as follows:
start with any inner product $\langle{\cdot,\cdot}\rangle$ on $\R^n$,
and define the inner product
\[  [[u,v]]
:=  \int_G \langle{gu,gv}\rangle\,d\mu(g).
\]
This is an invariant inner product on $G$, because
\[  [[hu,hv]]
=   \int_G \langle{hgu,hgv}\rangle\,d\mu(g)
=   \int_G \langle{gu,gv}\rangle\,d\mu(g)
=   [[u,v]],
\]
where the second equality follows again from the left invariance of $\mu$.
We have $[[u,u]]$ non--negative and equal to zero if and only if $u=0$, because
the same thing is true for $\langle{u,u}\rangle$.
\end{remark}

We have seen that convex bodies give rise to compact affine groups through their
automorphism groups. Conversely, it is a well known fact in Lie group theory that
a compact Lie group can be imbedded as a closed subgroup of the linear group $GL(V)$
for some finite dimensional vector space $V$. This can be found in most books on
Lie groups, see for example \cite{Bump04}, \cite{OnishchickVinberg90}, \cite{Zelobenko73}.
If $G$ is a compact, affine Lie group in $\R^n$, then it is isomorphic to the compact
linear group
\[  \left\{\bar{T}=\begin{pmatrix}A &a\\0&1\end{pmatrix}: T\in G,\;T(x)=a+Ax\right\} \]
in $\R^{n+1}$. In fact, if $T(x)=a+Ax$, then
\[  \bar{T}\begin{pmatrix}x\\1\end{pmatrix}
=   \begin{pmatrix}a+Ax\\1\end{pmatrix}.
\]
Thus, the affine transformations of $\R^n$ are in one--to--one correspondence with the
linear transformations of $\R^{n+1}$ that keep the hyperplane $\{(x,1): x\in\R^n\}$ invariant.
If $G$ is a compact linear group in $\R^n$, then one can consider the convex set $K=\co(Gx)$
where $x$ is an arbitrary point in $\R^n$.  Clearly, $K$ is invariant under $G$, that is,
$G\subseteq\Aut(K)$.  In this way, one can obtain the following purely group theoretical result
employing one of the methods in Remark~\ref{Remark:invariance}, see \cite{DanzerLaugwitzLenz57},
\cite{OnishchickVinberg90}, \cite{Zelobenko73}:

\begin{theorem}
\label{Theorem:compact-linear-groups}
Let $G$ be a compact group of linear transformations on $\R^n$.  The group $G$ fixes a point in $\R^n$,
that is, there exists $c$ in $\R^n$ such that $gc=c$ for all $g$ in $G$.  Moreover, there exists a
positive definite matrix $X$ invariant under $G$, that is, $g^TXg=X$ for all $g$ in $G$.
\end{theorem}

We remark that the Haar measure approach in Remark~\ref{Remark:invariance} already proves this
result directly, without considering orbits and convex sets.
One can also extend the theorem to affine compact groups, for example by considering the isomorphic
linear compact group in $\R^{n+1}$. However, the method of employing convex bodies does have its
merits.  For example, a simple proof of the \emph{algebraicity} of compact linear groups can be
found in \cite{OnishchickVinberg90}, pp. 130--135 and \cite{Zelobenko73}, Chapter 15,
using this approach.

The existence of the invariant quadratic form for compact Lie groups is a very important result,
since it implies that its linear representations are \emph{completely reducible}, that is,
its representations can be written as direct products of \emph{irreducible representations}.
For the group $\Aut(K)$, this simply means that the matrices appearing
in the linear parts of the affine transformations in $\Aut(K)$ must all have the same block
diagonal structure.

We also mention that such concepts as the fixed points of $\Aut(K)$ as well as its
invariant quadratic forms belong to the \emph{invariant theory}.  For example,
if $G$ is a linear group in $GL(\R^n)$, the \emph{invariant polynomials} of $G$
is the set of polynomials
\[     \R[x_1,\ldots,x_n]^G
:=     \{p\in\R[x_1,\ldots,x_n]: p(gx)=p(x),\ \forall\, g\in G,\ \forall\, x\in\R^n\}.
\]
The determination of the invariant polynomials for specific groups is one of the major
goals of invariant theory whose origins go back to the works of Cayley, Sylvester, Gordan,
Hilbert, etc. in the 19th century.  A major result going back to Hilbert in 1890s implies
that $\R[x_1,\ldots,x_n]^{\Aut(K)}$ is \emph{finitely generated}.  This means that there
exists finitely many invariant polynomials $S$ (which can be assumed homogeneous) such that
every invariant polynomial in $\R[x_1,\ldots,x_n]^{\Aut(K)}$ can be written as a polynomial
of elements of $S$.  See \cite{Zelobenko73}, pp. 280--281 for a fairly simple, direct proof.
A direct significance of this result for the ellipsoid $\ce(K)=E(X,c)$, say, is that
both $c$ and $X$ are invariant. Consequently, if the number of generators is small, then
the complexity of finding the ellipsoid $\ce(K)$ simplifies. This will
be demonstrated in the remaining Sections of our paper.

The automorphism group proves to be useful in other ways as well.  For example, the following
result of Davies~\cite{Davies74} is quite interesting.  Davies calls a convex body
$K$ in $\R^n$ \emph{symmetric} if $\Aut(K)$ acts \emph{transitively} on the extreme points
of $K$, that is, given any two points $x,y$ in $\ext(K)$, there exists a transformation
$g$ in $\Aut(K)$ such that $gx=y$.  He then proves the following result

\begin{lemma}
\label{Lemma:Davies}
If $K=\co(Gx)$ is a symmetric convex body in $\R^n$ in the sense Davies, then there exists a 
\emph{unique} fixed point of $\Aut(K)$ in $K$.
\end{lemma}
\begin{proof}
Let $\mu$ be the Haar probability measure $\mu$ on the compact Lie group $G:=\Aut(K)$.
Let $x$ be an arbitrary point in $\ext(K)$.  Since $K$ is symmetric, $\ext(K)=Gx$. Define the point
\[ c := \int_G gxd\mu(g). \]
The point $c$ is invariant under the action of $G$, since if $h\in G$, we have
$hc=\int_G hgxd\mu(g)=\int_G gxd\mu(g)=c$, where the second equality is a consequence of the invariance
of $\mu$.  Now, if $a$ is any invariant point in $K$, then $ga=a$ for all $g$ in $G$, and we have
$a=\int_Ga\,d\mu(g)=\int_Gga\,d\mu(g)$.  By Minkowski theorem, we have $a\in\co(ext(K))=\co(Gx)$,
that is, $a=\sum_{i=1}^k\lambda_ig_ix$ for some $\{\lambda_i,g_i\}_1^k$ where $g_i$ in $G$ and $\lambda_i\geq0$,
$\sum_{i=1}^k\lambda_i=1$.  Therefore,
\[   a
=    \int_Gga\,d\mu(g)
=    \sum_{i=1}^k\lambda_i\int_Ggg_ix\,d\mu(g)
=    (\sum_{i=1}^k\lambda_i)\int_Ggx\,d\mu(g)
=    c,
\]
where third equality follows again from the invariance of the measure $\mu$.
\end{proof}

The lemma implies in particular that the center of gravity of a symmetric body $K$
is the only invariant point of $K$ under the action of $\Aut(K)$. Thus, the centers of
the ellipsoids $\ce(K)$ and $\ie(K)$ must coincide and be equal to the center of gravity of $K$.

It is also possible to determine a formula for the matrix $X$ in the circumscribing ellipsoid
$\ce(K)=E(X,c)$.

\begin{lemma}
Let $K\subset\R^n$ be a convex body, symmetric in the sense of Davies.  If $X$ is an invariant matrix
of $K$ such that $X^{-1}\in\co\left(\{(y-c)(y-c)^T: y\in\ext(K)\}\right)$ where $c$ is the invariant
point of $K$, then
\[   X^{-1}
=    \int_G g^{-1}(x-c)(x-c)^T(g^T)^{-1}\,d\mu(g),
\]
where $x$ is an arbitrary point in $\ext(K)$ and $\mu$ is the Haar probability measure on $K$.
\end{lemma}
\begin{proof}
Pick an arbitrary point $x$ in $\ext(K)$ such that $\ext(K)=Gx$. We may assume without any loss of
generality that $\cg(K)=0$.  Since $X$ is an invariant matrix, we have $g^TXg=X$ or
$X^{-1}=g^{-1}X^{-1}g^{-T}$ where we defined $g^{-T}:=(g^{-1})^T=(g^T)^{-1}$.
Let $X^{-1}=\sum_{i=1}^k\lambda_iy_iy_i^T=\sum_{i=1}^k\lambda_ih_ixx^Th_i^T$ where $h_i\in G$.
We have
\begin{equation*}
\begin{split}
     X^{-1}
&=   \int_G g^{-1}X^{-1}g^{-T}\,d\mu(g)
=    \sum_{i=1}^k\lambda_i\int_G (g^{-1}h_i)xx^T(g^{-1}h_i)^T\,d\mu(g)\\
&=   \sum_{i=1}^k\lambda_i\int_G g^{-1}xx^Tg^{-T}\,d\mu(g)
=    \int_G g^{-1}xx^Tg^{-T}\,d\mu(g),
\end{split}
\end{equation*}
where the last equality follows from the right invariance of the Haar measure.
\end{proof}

\begin{corollary}
Let $K$ be a convex body in $\R^n$ symmetric in the sense of Davies.
The extremal covering ellipsoid of $\ce(K)=E(X,c)$ has center $c=\cg(K)$ and
\[   X^{-1} =  n\int_G g^{-1}(x-c)(x-c)^Tg^{-T}\,d\mu(g),
\]
where $x\in\ext(K)$ is an arbitrary point and $\mu$ is the Haar probability measure on $K$.
\end{corollary}
This is an immediate consequence of the above lemma, Theorem~\ref{Theorem:FJ-conditions-for-MVCE-SIP},
and \eqref{eq:sum-of-lambdas-is-n}.

Many interesting questions and research directions remain regarding the automorphism groups of
convex bodies.  It is not practical to investigate these in this paper; doing so would
increase the size of the paper beyond reasonable bounds and also change its character.
We plan to pursue these issues in future papers.

%*************************************************************************************************

\section{Invariance properties of a slab}
\label{sec:invariance-properties-of-a-slab}

In this paper, one of the problems we are interested in is the determination of the extremal
ellipsoids of the convex body $K$  which is the part of a given ellipsoid $E(X_0,c_0)$ between
two parallel hyperplanes,
\begin{equation*}
    K = \{x : \langle{X_0(x-c_0),\,x-c_0}\rangle\leq1,\, a\leq\langle{p,\,x-c_0}\rangle\leq b\},
\end{equation*}
where $p$ is a non--zero vector in $\R^n$, and where $a$ and $b$ are such that $K$ is nonempty.
Recall that we call such a convex body $K$ a slab.

In this section, we determine $\Aut(K)$ and the form of the ellipsoids
$\ce(K)$ and $\ie(K)$.

If we substitute $u=X_0^{1/2}(x-c_0)$, that is, $x=c_0+X_0^{-1/2}u$,
the quadratic inequality $\langle{X_0(x-c_0),x-c_0}\rangle\leq\,1$ becomes
$||u||\leq1$ and making the further substitution $q:=X_0^{-1/2}p$,
the linear form $\langle{p,x-c_0}\rangle$ becomes
\[ \langle{p,x-c_0}\rangle
=  \langle{X_0^{-1/2}p,X_0^{1/2}(x-c_0)}\rangle
=  \langle{q,u}\rangle. \]
Defining $\overline{p}=q/||q||$, $\alpha=a/||q||$ and
$\beta=b/||q||$, the linear inequalities
$a\leq\langle{p,x-c_0}\rangle\leq b$ reduce to
$\alpha\leq\langle{\overline{p},u}\rangle\leq\beta$.
Altogether, these substitutions give
\[   K\,
=\   \{c_0+X_0^{-1/2}u : ||u||\leq1,\;
      \alpha\leq\langle{\overline{p},u}\rangle\leq\beta
     \}.
\]
Let $Q$ be an orthogonal matrix such that $Qe_1=\overline{p}$, where
$e_1=(1,0,\ldots,0)^T$ in $\R^n$. Defining $v=Q^{-1}u$, we finally have
\[   K\,
=\,  \{c_0+X_0^{-1/2}Qv : ||v||\leq1,\;
      \alpha\leq\langle{e_1,v}\rangle\leq\beta
     \}, \]
that is, $K=c_0+X_0^{-1/2}Q(\tilde{K})$, where
$\tilde{K}=\{v : ||v||\leq1,\, \alpha\leq\langle{e_1,v}\rangle\leq\beta\}$.

Since an affine transformation leaves ratios of volumes unchanged,
we assume from here on, without loss of any generality, that our initial convex
body $K$, which we denote by $\slab$, has the form
\begin{equation}
\label{eq:slab}
    \slab
=   \{x\in\R^n: ||x||\leq1,\;\alpha\leq x_1\leq\beta\},
\end{equation}
where $-1\leq\alpha<\beta\leq1$.

\begin{remark}
To simplify our proofs, we assume in this paper that $\beta^2\geq\alpha^2$. We can
always achieve this by working with the convex body $-\slab$ instead of $\slab$,
if necessary.
\end{remark}

We use the symmetry properties of $\slab$ to determine the possible
forms of its extremal ellipsoids. This idea seems to be first suggested in
\cite{KonigPallaschke81} for determining $\ce(\slab)$.

\begin{lemma}
\label{Lemma:Aut(slab)}
If $\alpha=-1$ and $\beta=1$, then the automorphism group of the
slab $\slab=B_n$ consists of the $n\times n$ orthogonal matrices.  In the remaining
cases, the automorphism group $\Aut(\slab)$ consists of linear transformations $T(u)=Au$
where $A$ is a matrix of the form
\[     A
=      \begin{bmatrix} a_{11}&0\\0&\bar{A}\end{bmatrix},
        \quad  a_{11}\in\R,\; \bar{A}\in O_{n-1},
\]
with $a_{11}=1$ if $\alpha\neq-\beta$ and $a_{11}=\pm1$ if $\alpha=-\beta$.
\end{lemma}
\begin{proof}
It is proved in Lemma~\ref{Lemma:Aut(E)} that $\Aut(B_n)=O_n$, so we consider
the remaining cases.

Let $T(x)=a+Au$ be an automorphism of $\slab$.  We write
$A=\begin{bmatrix}a_{11}&c^T\\b&\bar{A}\end{bmatrix}$ and $a=(a_1,\bar{a})$,
where $a_{11}$ and $a_1$ are scalars and the rest of the variables have the
appropriate dimensions.  Since $T$ is an invertible affine map,
$T(\ext(\slab))=\ext(\slab)$, where $\ext(\slab)$ is the set of extreme
points of $\slab$ given by
\[    \ext(\slab)
=     \left\{u=(u_1,(1-u_1^2)^{1/2}\bar{u})\in\R\times\R^{n-1}:
      \alpha\leq u_1\leq\beta, \,||u||=1, \,||\bar{u}||=1\right\}.
\]
If $u=(u_1,(1-u_1^2)^{1/2}\bar{u})$ is in $\ext(\slab)$ with $||\bar{u}||=1$, then
\[     w
:=     a+Au
=      \begin{bmatrix}
           a_{11}u_1+(1-u_1^2)^{1/2}\langle{c,\bar{u}}\rangle+a_1\\
           u_1b+(1-u_1^2)^{1/2}\bar{A}\bar{u}+\bar{a}
       \end{bmatrix}.
\]
We have $||w||=1$, that is,
\begin{equation}
\label{eq:||w||=1}
\begin{split}
       1
&=     (a_{11}u_1+a_1)^2
        +(1-u_1^2)\langle{c,\bar{u}}\rangle^2
         +2(1-u_1^2)^{1/2}\langle{c,\bar{u}}\rangle(a_{11}u_1+a_1)\\
&\quad    \ +||u_1b+\bar{a}||^2
            +(1-u_1^2)||\bar{A}\bar{u}||^2
             +2(1-u_1^2)^{1/2}\langle{\bar{A}\bar{u}, u_1b+\bar{a}}\rangle\\
&=     \left\lbrace
          (a_{11}u_1+a_1)^2+||u_1b+\bar{a}||^2
       \right\rbrace
        +(1-u_1^2)
          \left\langle
             (\bar{A}^T\bar{A}+cc^T)\bar{u}, \bar{u}
          \right\rangle\\
&\quad   \ +2(1-u_1^2)^{1/2}
           \left\langle
              (a_{11}u_1+a_1)c+\bar{A}^T(u_1b+\bar{a}), \bar{u}
           \right\rangle\\
&=:     q(\bar{u}),\quad\forall ||\bar{u}||=1.
\end{split}
\end{equation}
Fix $u_1\in(\alpha,\beta)$, so that $1-u_1^2\neq0$.  The argument used in
the proof of Lemma~\ref{Lemma:Aut(E)} shows that
\begin{equation}
\label{eq:linear-term-of-q(tilde{u})-vanishes}
       (a_{11}u_1+a_1)c+\bar{A}^T(u_1b+\bar{a})
=      0,\quad\forall\, u_1\in(\alpha,\beta),
\end{equation}
and that
$      (1-u_1)^2
       \langle
          (\bar{A}^T\bar{A}+cc^T)\bar{u}, \bar{u}
       \rangle
=      \left\lbrace
          1-(a_{11}u_1+a_1)^2-||u_1b+\bar{a}||^2
       \right\rbrace
       ||\bar{u}||^2$ for all $\bar{u}\in\R^{n-1},
$
that is,
$      (1-u_1)^2
       (\bar{A}^T\bar{A}+cc^T)
=      \left\lbrace
           1-(a_{11}u_1+a_1)^2-||u_1b+\bar{a}||^2
       \right\rbrace
       I_{n-1},
$
for all $u_1\in(\alpha,\beta)$.  Therefore, there exists a constant $k$ such that
\begin{equation*}
\begin{split}
       kI_{n-1}
&=     \bar{A}^T\bar{A}+cc^T,\\
       0
&=     (a_{11}u_1+a_1)^2+||u_1b+\bar{a}||^2+k(1-u_1^2)-1,
        \quad\forall \,u_1\in(\alpha,\beta).
\end{split}
\end{equation*}
The equation \eqref{eq:linear-term-of-q(tilde{u})-vanishes} implies the first two equations
in \eqref{eq:final-cond-first-set}, while the equation above gives rest of the equations below,
\begin{align}
\label{eq:final-cond-first-set}
    0&=\bar{A}^Tb+a_{11}c, & 0&=\bar{A}^T\bar{a}+a_1c, & kI_{n-1}&=\bar{A}^T\bar{A}+cc^T,\\
\label{eq:final-cond-second-set}
    0&=a_{11}^2+||b||^2-k, & 0&=a_1a_{11}+\langle{b,\bar{a}}\rangle, & 0&=a_1^2+||\bar{a}||^2+k-1.
\end{align}
We have, therefore,
\begin{equation*}
      A^TA
=     \begin{bmatrix}
         a_{11}^2+||b||^2     &a_{11}c^T+b^T\bar{A}\\
         a_{11}c+\bar{A}^Tb &\bar{A}^T\bar{A}+cc^T
      \end{bmatrix}
=     \begin{bmatrix}k &0\\0 &kI_{n-1}\end{bmatrix}
=     kI_n.
\end{equation*}
Since $|\det A|=1$ and $A^TA$ is positive semidefinite, we have $k=1$.
This proves that $A$ is an orthogonal matrix.  Furthermore,
the last equation in \eqref{eq:final-cond-second-set} gives $a=(a_1,\bar{a})=0$.

Let $x$ be in $\slab$. As $||Ax||=||x||\leq1$ and
$\langle{e_1,x}\rangle=\langle{Ae_1,Ax}\rangle$, we have
\[  \slab
=   A\slab
=   \{Ax: ||x||\leq1,\; \alpha
    \leq\langle{e_1,x}\rangle\leq\beta\}
=   \{x: ||x||\leq1,\; \alpha\leq\langle{Ae_1,x}\rangle
    \leq\beta\}.
\]
If $\alpha\neq-\beta$, then we must have $Ae_1=e_1$, and if $\alpha=-\beta$,
then $\slab=-\slab$ and we have $Ae_1=\pm e_1$.
Since $Ae_1=\begin{pmatrix}a_{11}\\b\end{pmatrix}$, we see that $|a_{11}|=1$ and
$b=c=0$.  It is then clear that $\bar{A}$ belongs to $O_{n-1}$.

Conversely, it is easy to verify that any matrix $A$ in the form above is in $\Aut(K)$.
\end{proof}

\begin{lemma}
\label{Lemma:invariance}
The extremal ellipsoids $\ce(\slab)$ and $\ie(\slab)$ have the form $E(X,c)$ where
$c=\tau e_1$ and $X=\diag(a,b,...,b)$ for some $a>0$, $b>0$ and $\tau$ in $\R$.
Moreover, if $\alpha=-\beta$, then $c=0$.
\end{lemma}
\begin{proof}
Since the proofs are the same, we only consider the ellipsoid $\ce(\slab)$.
Let $U=\begin{bmatrix}1&0\\0&\bar{U}\end{bmatrix}$ be in $\Aut(K)$. Write
$c=(c_1,\bar{c})$. It follows from Theorem~\ref{Theorem:invariance} that $Uc=c$.
This implies that $\bar{U}\bar{c}=\bar{c}$ for all $\bar{U}$ in $O_{n-1}$.
Choosing $\bar{U}=-I_{n-1}$, we obtain $\bar{c}=0$.  If $\alpha=-\beta$, then
choosing $U=-I_n$ in $O_n$ gives $c=0$.

Let $\ce(\slab)=T(B_n)=E(X,c)$ where $T(x)=c+X^{-1/2}x$.  Lemma~\ref{Lemma:Aut(E)} implies that
$U^TXU=X$, and writing $X=\begin{bmatrix}x_{11}&v^T\\v&\bar{X}\end{bmatrix}$, this equation
gives $\bar{U}^Tv=v$ and $\bar{U}^T\bar{X}\bar{U}=\bar{X}$ for all
$\bar{U}$ in $O_{n-1}$.  The first equation implies $v=0$. In the second equation, we can
choose $\bar{U}$ so that the left hand side is a diagonal matrix, proving that
$\bar{X}$ itself must be a diagonal matrix. If $\bar{U}$ is the permutation matrix
switching columns $i$ and $j$, then the equation $\bar{U}^T\bar{X}\bar{U}=\bar{X}$
gives $\bar{X}_{ii}=\bar{X}_{jj}$.  This proves that $X=\diag(a,b,b,\ldots,b)$
for some $a>0,\,b>0$.
\end{proof}

%*************************************************************************************************

\section{Determination of the minimum volume circumscribed ellipsoid of a slab}
\label{sec:MVCE-of-a-slab}

In this section, we give explicit formulae for the minimum volume circumscribed
ellipsoid of the slab $\slab$ in \eqref{eq:slab} using the Fritz John optimality
conditions \eqref{eq:FJ-for-MVCE-general} and Lemma~\ref{Lemma:invariance}.
In this section, $K$ will always denote the convex body $\slab$.

The following theorem is one of our main results in this paper.

\begin{theorem}
\label{Theorem:optimal-solution-for-slab-MVCE-problem}
The minimum volume circumscribed ellipsoid $\ce(\slab)$ has the form $E(X,c)$ where
$c=\tau e_1$ and $X=\emph{diag}(a,b,\ldots,b)$, where the
parameters $a>0$, $b>0$, and $\alpha<\tau<\beta$ are given as follows:\\
(i) If $\alpha\beta\,\leq\,-1/n$, then
\begin{equation}
    \label{eq:Solution-i}
    \tau=0,\,\text{ and }\,a=b=1.
\end{equation}
(ii) If $\alpha+\beta\,=\,0$ and $\alpha\beta\,>\,-1/n$, then
\begin{equation}
    \label{eq:Solution-ii}
    \tau=0,\quad
    a=\frac{1}{n\beta^2},\quad
    b=\frac{n-1}{n(1-\beta^2)}.
\end{equation}
(iii) If $\alpha+\beta\,\neq0$ and $\alpha\beta\,>\,-1/n$, then
\begin{equation}
\label{eq:Solution-iii}
\begin{split}
    \tau
&=  \frac{n(\beta+\alpha)^2+2(1+\alpha\beta)-
    \sqrt{\Delta}}{2(n+1)(\beta+\alpha)},\\
    a
&=  \frac{1}{n(\tau-\alpha)(\beta-\tau)},\quad
    b
=   \frac{1-a(\tau-\alpha)^2}{1-\alpha^2},
\end{split}
\end{equation}
where $\Delta=n^2(\beta^2-\alpha^2)^2+4(1-\alpha^2)(1-\beta^2)$.
\end{theorem}

We remark that Corollary~\ref{Corollary:diameter-breadth} also implies
the converse of Part (i) in the theorem.

\begin{proof}
Lemma~\ref{Lemma:invariance} implies that $X=diag(a,b,...,b)$ and $c=\tau e_1$.
Writing $u_i=(y_i,z_i)\in\R\times\R^{n-1}$, $i=1,\ldots,k$, where
$1=||u_i||^2=y_i^2+||z_i||^2$, that is $||z_i||^2=1-y_i^2$, and noting that
$  (u_i-c)(u_i-c)^T
=  \begin{bmatrix}
      (y_i-\tau)^2  &(y_i-\tau)z_i^T\\
      (y_i-\tau)z_i &z_iz_i^T
   \end{bmatrix}$,
the Fritz John (necessary and sufficient) optimality conditions
\eqref{eq:FJ-for-MVCE-general} may be written in the form
\begin{eqnarray}
           \tau
&=&\
           \frac{1}{n}\sum_{i=1}^k\lambda_iy_i,\quad
           0
=
           \sum_{i=1}^k\lambda_iz_i,\quad
           \sum_{i=1}^k\lambda_i
=          n,\label{eq:FJ-for-slab-MVCE-1} \\
           \frac{1}{a}
&=&\
           \sum_{i=1}^k\lambda_i(y_i-\tau)^2,\quad
           0
=
           \sum_{i=1}^k\lambda_i(y_i-\tau)z_i,\quad
           \frac{1}{b}I_{n-1}
=          \sum_{i=1}^k\lambda_iz_iz_i^T,
           \label{eq:FJ-for-slab-MVCE-2}\\
           0
&=&\
           a(y_i-\tau)^2+b(1-y_i^2)-1,\quad i=1,\ldots,k,
           \label{eq:FJ-for-slab-MVCE-3}\\
           0
&\geq&\    a(y-\tau)^2+b(1-y^2)-1,\quad
            \forall\, y\in[\alpha,\beta].
           \label{eq:FJ-for-slab-MVCE-4}
\end{eqnarray}
The last line expresses the feasibility condition $K\subseteq E(X,c)$: any point
$x=(y,z)$ satisfying $\alpha\leq y\leq\beta$ and $||x||=1$ lies in $K$, hence in
$E(X,c)$, so that it satisfies the conditions $y^2+||z||^2=1$ and
$a(y-\tau)^2+b||z||^2\leq1$.

The conditions \eqref{eq:FJ-for-slab-MVCE-1}--\eqref{eq:FJ-for-slab-MVCE-4} thus
characterize the ellipsoid $\ce(K)$ for $K=\slab$ in \eqref{eq:slab}.  Since the
ellipsoid $\ce(K)$ is unique, its paremeters $(\tau,a,b)$ are unique and can be
recovered from the above conditions.  These are done in the technical lemmas below.
\end{proof}

\begin{lemma}
\label{Lemma:a=b}
If $a=b$ in the ellipsoid $\ce(\slab)$, then $\tau=0$, $a=b=1$,
and $\alpha\beta\leq-1/n$.
\end{lemma}
\begin{proof}
Since $a=b$, \eqref{eq:FJ-for-slab-MVCE-3} gives the equation
$2a\tau y_i=a\tau^2+a-1$. We have $\tau=0$, since otherwise all $y_i$
are the same, and the first and third equations in \eqref{eq:FJ-for-slab-MVCE-1}
imply that $y_i=\tau$, contradicting the first equation in
\eqref{eq:FJ-for-slab-MVCE-2}. The equation $2a\tau y_i=a\tau^2+a-1$
reduces to $a=1=b$. Finally, since $\alpha\leq y_i\leq\beta$, we obtain
\begin{equation*}
      0
\geq
      \sum_{i=1}^k\lambda_i(y_i-\alpha)(y_i-\beta)
=
      \sum_{i=1}^k\lambda_iy_i^2
       -(\alpha+\beta)\sum_{i=1}^k\lambda_iy_i
        +\alpha\beta\sum_{i=1}^k\lambda_i
=     1+n\alpha\beta,
\end{equation*}
where the last equation follows since $\sum_{i=1}^k\lambda_iy_i^2=1$
from the first equation in \eqref{eq:FJ-for-slab-MVCE-2} and
$\sum_{i=1}^k\lambda_iy_i=0$, $\sum_{i=1}^k\lambda_i=n$ from
\eqref{eq:FJ-for-slab-MVCE-1}.
\end{proof}

\begin{lemma}
\label{Lemma:a-neq-b-implies-u-equals-a-or-b}
If $a\not=b$ in the ellipsoid $\ce(\slab)$, then $a>b$ and the leading coordinate $y_i$
of a contact point must be $\alpha$ or $\beta$, and both values are taken.
\end{lemma}
\begin{proof}
Observe that the function $g(y):=a(y-\tau)^2+b(1-y^2)-1$ in
\eqref{eq:FJ-for-slab-MVCE-4} is nonpositive on the interval $I=[\alpha,\beta]$
and equals zero at each $y_i$. We claim that $y_i$ can not take a single value:
otherwise the first and third equations in \eqref{eq:FJ-for-slab-MVCE-1}
imply that $y_i=\tau$, contradicting the first equation in \eqref{eq:FJ-for-slab-MVCE-2}.
(This result also follows from
Corollary~\ref{Corollary:contact-points-of-MVCE-cannot-lie-in-any-halfspace}.)
Since $g$ is a quadratic function, $y_i$ must take exactly two values, and
\eqref{eq:FJ-for-slab-MVCE-4} implies that these two values must coincide with the
endpoints of the interval $I$. Furthermore, $g(y)\leq0$ only on $I$, has a global
minimizer there, and so it must be a strictly convex function.  This proves that $a>b$.
\end{proof}

\begin{lemma}
\label{Lemma:a-neq-b}
If $a\not=b$ in the ellipsoid $\ce(\slab)$, then $(\tau,a,b)$ are given by the equations
\eqref{eq:Solution-ii}  and \eqref{eq:Solution-iii}. Moreover, $\alpha\beta>-1/n$.
\end{lemma}
\begin{proof}
Lemma~\ref{Lemma:a-neq-b-implies-u-equals-a-or-b} and equation
\eqref{eq:FJ-for-slab-MVCE-3} give $a(\beta-\tau)^2+b(1-\beta^2)=1$  and
$a(\alpha-\tau)^2+b(1-\alpha^2)=1$. Subtracting the second equation from the first
and dividing by $\beta-\alpha\neq0$ yields the equation $\tau=(1-b/a)(\alpha+\beta)/2$.
We also have
\begin{equation*}
\begin{split}
      0
&=    \sum_{i=1}^k\lambda_i(y_i-\alpha)(y_i-\beta)
=     \sum_{i=1}^k\lambda_i[(y_i-\tau)-(\alpha-\tau)]\cdot[(y_i-\tau)-(\beta-\tau)]\\
&=    \sum_{i=1}^k\lambda_i(y_i-\tau)^2-(\alpha+\beta-2\tau)\sum_{i=1}^k\lambda_i(y_i-\tau)
        +(\alpha-\tau)(\beta-\tau)\sum_{i=1}^k\lambda_i\\
&=    \frac{1}{a}+n(\alpha-\tau)(\beta-\tau),
\end{split}
\end{equation*}
where the first equation follows from Lemma~\ref{Lemma:a-neq-b-implies-u-equals-a-or-b}, the
last equation from \eqref{eq:FJ-for-slab-MVCE-1} and \eqref{eq:FJ-for-slab-MVCE-2}.

Altogether, we have the equations
\begin{equation}
\label{eq:a-b-tau}
    1
=  a(\alpha-\tau)^2+b(1-\alpha^2),\quad
    \tau
=   \left(1-\frac{b}{a}\right)\cdot\frac{\alpha+\beta}{2},\quad
    \frac{1}{a}
=  n(\tau-\alpha)(\beta-\tau),
\end{equation}
which we use to compute the variables $a$, $b$, and $\tau$.

If $\alpha=-\beta$, then the second equation above gives $\tau=0$.
Then the third and first equations in \eqref{eq:a-b-tau} give
$a=1/(n\alpha^2)$ and $b=(n-1)/(n(1-\alpha^2))$, respectively. Lastly,
Lemma~\ref{Lemma:a-neq-b-implies-u-equals-a-or-b} gives $a>b$,
and this implies $1+n\alpha\beta>0$.

If $\beta\neq-\alpha$, then the first and
third equations in \eqref{eq:a-b-tau}
give
$(\alpha-\tau)^2+(b/a)(1-\alpha^2)=n(\tau-\alpha-)(\beta-\tau)$
and the second equation gives $b/a=1-2\tau/(\alpha+\beta)$.
Substituting this value of $b/a$ in the preceding one
leads to the quadratic equality for $\tau$,
\begin{equation}
\label{eq:quadratic-equation-for-tau}
    (n+1)(\alpha+\beta)\tau^2
     -\left(n(\alpha+\beta)^2+2(1+\alpha\beta)\right)\tau
      +(\alpha+\beta)(1+n\alpha\beta)=0.
\end{equation}
A straightforward but tedious calculation shows that
the discriminant is
$\Delta=n^2(\beta^2-\alpha^2)^2+4(1-\beta^2)(1-\alpha^2)>0$.
We claim that the feasible root is the one with negative discriminant.
If $\overline{\tau}$ is the root with positive discriminant, then
$   \overline{\tau}-\beta
=   [n(\alpha^2-\beta^2)+2(1-\beta^2)+
    \sqrt{\Delta}]/(2(n+1)(\alpha+\beta))\geq0$.
Recalling that $\beta^2\geq\alpha^2$, we have
$[n(\alpha^2-\beta^2)+2(1-\beta^2)]^2-\Delta
=4(n+1)(1-\beta^2)(\alpha^2-\beta^2)\leq0$.
This gives $\overline{\tau}\geq\beta$,
proving the claim, as we must have
$\alpha<\tau<\beta$.  Therefore,
\begin{equation*}
\begin{split}
     \tau
&=   \frac{n(\alpha+\beta)^2+2(1+\alpha\beta)-\sqrt{\Delta}}
       {2(n+1)(\alpha+\beta)},\\
     a
&=   \frac{1}{n(\tau-\alpha)(\beta-\tau)},\quad
     b
=    \frac{1-a(\alpha-\tau)^2}{1-\alpha^2},
\end{split}
\end{equation*}
where the equations for $a$ and $b$ follow from the first and second
equations in \eqref{eq:a-b-tau}.

Finally, $\tau=(1-b/a)(\alpha+\beta)/2$ from \eqref{eq:a-b-tau} and
Lemma~\ref{Lemma:a-neq-b-implies-u-equals-a-or-b} gives $a>b$, implying
$\tau>0$.  From the formula above for $\tau$, we get
\[
    0
<
    (n(\alpha+\beta)^2+2(1+n\alpha\beta))^2-\Delta
=
    4(n+1)(\alpha+\beta)^2(1+n\alpha\beta).
\]
This gives $1+n\alpha\beta>0$. The lemma is proved.
\end{proof}

\begin{remark}
\label{Remark:concentration-of-measure}
Some of the results contained in Theorem~\ref{Theorem:optimal-solution-for-slab-MVCE-problem}
may seem very counter--intuitive.  For example, consider the slab $\slab$ when
$-\alpha=\beta=1/\sqrt{n}$.  Although the width of this slab is $2/\sqrt{n}$, very small
for large $n$, the optimal covering ellipsoid is the unit ball. This seemingly improbable
behavior may be explained by the \emph{concentration of measure} phenomenon: most of the
volume of a high dimensional ball is concentrated in a thin strip around the ``equator'',
see for example \cite{Ball97}. There is a sizable literature on concentration of measure;
the interested reader may consult the reference \cite{Ledoux01}  for more
information on this important topic.
\end{remark}

%*************************************************************************************************

\subsection{Determination of the covering ellipsoid by nonlinear programming}
\label{subsec:NLP-min}

In this section, we give a proof of
Theorem~\ref{Theorem:optimal-solution-for-slab-MVCE-problem} which is completely
independent of the previous one based on semi--infinite programming.
The proof uses the uniqueness of the covering ellipsoid, Lemma~\ref{Lemma:invariance}
on the form of the optimal ellipsoid, and
Corollary~\ref{Corollary:contact-points-of-MVCE-cannot-lie-in-any-halfspace}.
We thus need proofs of the first and the last results that do not depend on the results of
\S\ref{sec:MVCE} and Theorem~\ref{Theorem:optimal-solution-for-slab-MVCE-problem}.
We note that an elementary and direct proof of the uniqueness of the optimal covering
ellipsoid can be found, for example, in Danzer et al.~\cite{DanzerLaugwitzLenz57},
and we supply an independent, direct proof of
Corollary~\ref{Corollary:contact-points-of-MVCE-cannot-lie-in-any-halfspace}.
This last result is not strictly necessary, but it simplifies our proofs, and
it may be of independent interest.

Recall that Corollary~\ref{Corollary:contact-points-of-MVCE-cannot-lie-in-any-halfspace}
states that the contact points of an extremal covering ellipsoid cannot lie any half space
whose bounding hyperplane passes through the center of the ellipsoid. The following is an
independent proof of this fact, in the spirit of the proof in Grunbaum~\cite{Grunbaum64}
for the maximum volume inscribed ellipsoid.

\begin{proof}
We assume without loss of generality that the ellipsoid is the unit ball $B_n$.
Clearly, it suffices to show that the open halfspace $B^+:=\{x: x_n>0\}$
contains a point of $E\cap\partial K$.  We prove this by contradiction.

Consider the ellipsoids
$E(\lambda)=\{x:\, f(x)=a\sum_{i=1}^{n-1}x_i^2+b(x_n+\lambda)^2\leq1\}$,
having on their boundary the points $\{e_i\}_1^{n-1}$ and $-e_n=(0,0,\ldots,0,-1)$.
We have $b=1/(1-\lambda)^2$, $a=1-\lambda^2/(1-\lambda)^2=(1-2\lambda)/(1-\lambda)^2$,
and
\[   \vol(E(\lambda))
=    \left(a^{n-1}b\right)^{-1}
=    \frac{(1-\lambda)^{2n}}{(1-2\lambda)^{n-1}}.
\]
We claim that $K\subseteq E(\lambda)$ for small enough $\lambda>0$.
On the one hand, if $||x||=1$ and $x_n\leq0$, we have
$f(x)-1=\frac{2\lambda}{(1-\lambda)^2}(x_n^2+x_n)\leq0$. On the other hand,
since $B^+$ contains no contact points, there exists $\epsilon>0$
such $||x||<1-\epsilon$ for all $x\in K\cap B^+$. It follows by continuity that
$K\cap B^+\subset E(\lambda)$ for small enough $\lambda>0$. These prove the claim.

Lastly, $\vol(E(\lambda))<1$, since $(1-\lambda)^{2n}-(1-2\lambda)^{n-1}
=[1-2n\lambda+o(\lambda)]-[1-(n-1)(-2\lambda)+o(\lambda)]
=-2\lambda+o(\lambda)<0$ for small $\lambda>0$.
\end{proof}

\begin{theorem}
\label{Theorem:MVCE-for-slab-using-NLP}
The minimum volume covering ellipsoid problem for the slab $\slab$ can be formulated
as the nonlinear programming problem
\begin{equation}
\label{problem:min-NLP-for-MVCE-of-slab}
\begin{split}
  \min\quad  &-\text{ln }a-(n-1)\text{ln }b,\\
  \st\quad   &a\tau-(\frac{\alpha+\beta}{2})(a-b)=0,\\
             &a\tau^2+b-1-\alpha\beta(a-b)=0,\\
             &-a+b\leq0,
\end{split}
\end{equation}
whose solution is the same as the one given in
Theorem~\ref{Theorem:optimal-solution-for-slab-MVCE-problem}.
\end{theorem}
\begin{proof}
It follows from Lemma~\ref{Lemma:invariance} that the optimal ellipsoid $E(X,c)$ has
the form $X=\diag(a,b,\dotsc,b)$ and $c=\tau e_1$. Thus, the feasibility condition
$B_{\alpha\beta}\subseteq E(X,c)$ translates into the condition that the quadratic
function
\[    g(u)
=     a(u-\tau)^2+b(1-u^2)-1
\]
is non--positive on the interval $I=[\alpha,\beta]$. Furthermore,
Corollary~\ref{Corollary:contact-points-of-MVCE-cannot-lie-in-any-halfspace}
implies that there exist at least two contact points, which translates into the condition
that the quadratic function $g(u)$ takes the value zero at two distinct points in
the interval $I$.  A moment's reflection shows that $g(u)$ must take the value
zero at the endpoints $\alpha$ and $\beta$. Consequently, the function $g$ is a
non--negative multiple of the function $(u-\alpha)(u-\beta)$, that is
\[  g(u)+\mu(u-\alpha)(\beta-u)=0,\quad\text{for some}\quad \mu\geq0, \]
giving $a-b=\mu\geq0$, $(\alpha+\beta)\mu-2a\tau=0$, and $a\tau^2+b-1-\alpha\beta\mu=0$.
If we eliminate $\mu$ from these constraints, we arrive at the optimization
problem \eqref{problem:min-NLP-for-MVCE-of-slab}.  We have for it the Fritz John
optimality conditions (for ordinary nonlinear programming)
\begin{equation}
\label{FJ-conds-for-NLP}
\begin{split}
    \lambda_1(\tau-\frac{\alpha+\beta}{2})+
    \lambda_2(\tau^2-\alpha\beta)-\lambda_3
&=
    \frac{\lambda_0}{a},\\
    \lambda_1(\frac{\alpha+\beta}{2})+
    \lambda_2(1+\alpha\beta)+\lambda_3
&=
    \frac{\lambda_0(n-1)}{b},\\
    \lambda_1+2\lambda_2\tau
&=  0,
\end{split}
\end{equation}
for some $(\lambda_0,\lambda_1,\lambda_2,\lambda_3)\neq0$, $\lambda_0\geq0$, $\lambda_3\geq0$,
and satisfying $\lambda_3(a-b)=0$.

Adding the first two equations above and substituting
$\lambda_1=-2\lambda_2\tau$ from the third equation gives
\[ \lambda_2(1-\tau^2)
=  \frac{\lambda_0}{a}+\frac{\lambda_0(n-1)}{b}.
\]
If $\lambda_0=0$, we would have $\lambda_2(1-\tau^2)=0$, and since $\tau\neq\pm1$,
$\lambda_2=0$, and eventually $\lambda_1=0=\lambda_3$, a contradiction.
Thus, we may assume that $\lambda_0=1$.

Then the above equation gives
\begin{equation}
\label{eq:lambda-tau}
    \lambda_2(1-\tau^2)
=   \frac{1}{a}+\frac{n-1}{b}.
\end{equation}

We solve for the decision variables $(a,b,\tau)$ discussing separately the cases
$a=b$ and $a\neq b$ in the above optimality conditions. If $a=b$, then the first
constraint in \eqref{problem:min-NLP-for-MVCE-of-slab} gives $\tau=0$ and the second
constraint gives $b=1$.  It remains to prove that $\alpha\beta\leq-1/n$.
The equation \eqref{eq:lambda-tau} gives $\lambda_2=n$, and
$\lambda_1=-2\lambda_2\tau=0$. Substituting these in the first equation in
\eqref{FJ-conds-for-NLP} yields $0\leq\lambda_3=-n\alpha\beta-1$, that is,
$\alpha\beta\leq-1/n$.

We now consider the case $a\neq b$ but $\alpha+\beta=0$.  The first constraint
in problem \eqref{problem:min-NLP-for-MVCE-of-slab} gives $\tau=0$ and the third
one gives $\lambda_3=0$. Consequently, the first two conditions in \eqref{FJ-conds-for-NLP}
can be written as $\lambda_2\beta^2a=1$ and $\lambda_2(1-\beta^2)b=n-1$, respectively,
and the second constraint in \eqref{problem:min-NLP-for-MVCE-of-slab} gives
$\beta^2a+(1-\beta^2)b=1$.  These imply $\lambda_2=n$, and $a=1/n\beta^2$,
$b=\frac{n-1}{n(1-\beta^2)}$. Lastly, the condition $a>b$ gives $\alpha\beta>-1/n$.

Finally, we treat the case $a>b$ and $\alpha+\beta\neq0$. Again we have $\lambda_3=0$ and
\begin{equation*}
    (n-1)\frac{-\tau^2+(\alpha+\beta)\tau-\alpha\beta}
    {-(\beta+\alpha)\tau+(1+\alpha\beta)}
=
    \frac{b}{a}
=
    \frac{\beta+\alpha-2\tau}{\beta+\alpha}.
\end{equation*}
Here the first equality is obtained by dividing the first equation
in \eqref{FJ-conds-for-NLP} by the second one and substituting
$\lambda_1=-2\lambda_2\tau$, and the second equality follows
from the first constraint in problem \eqref{problem:min-NLP-for-MVCE-of-slab}.
Consequently, $\tau$ satisfies the quadratic equality
\begin{equation}
\label{eq:eq-for-tau-in-NLP}
    (n+1)(\alpha+\beta)\tau^2
    -[n(\alpha+\beta)^2+2(1+\alpha\beta)]\tau
    +(\alpha+\beta)(1+n\alpha\beta)
=
    0,
\end{equation}
which is the same equation as \eqref{eq:quadratic-equation-for-tau}
in the proof of Lemma \ref{Lemma:a-neq-b}. Following similar arguments, we find that
\begin{equation*}
\begin{split}
     \tau
&=   \frac{n(\alpha+\beta)^2+2(1+\alpha\beta)-\sqrt{\Delta}}
       {2(n+1)(\alpha+\beta)}.
\end{split}
\end{equation*}
Solving the first and the second constraints in
\eqref{problem:min-NLP-for-MVCE-of-slab} for $a$, say by Cramer's rule, we find
\begin{equation*}
    a
=
    \frac{\alpha+\beta}
    {(\alpha+\beta)(\tau^2+1)-2\tau(1+\alpha\beta)}.
\end{equation*}
It follows from \eqref{eq:eq-for-tau-in-NLP} that the denominator on the
right hand side of the expression above equals
\begin{equation*}
    n(\alpha+\beta)^2\tau-n(\alpha+\beta)\tau^2-
    n(\alpha+\beta)\alpha\beta
=   n(\alpha+\beta)(\tau-\alpha)(\beta-\tau).
\end{equation*}
This gives
\begin{equation*}
    a
=
    \frac{1}{n(\tau-\alpha)(\beta-\tau)},\qquad
    b
=
    \frac{\alpha+\beta-2\tau}{\alpha+\beta}a.
\end{equation*}
Finally, the inequality $\alpha\beta>-1/n$ follows from the same argument
at the end of the proof of Lemma \ref{Lemma:a-neq-b}.
\end{proof}

%*************************************************************************************************

\section{Determination of the maximum volume inscribed ellipsoid of a slab}
\label{sec:MVIE-of-a-slab}

In this section, we give explicit formulae for the maximum volume
inscribed ellipsoid of the slab $\slab$ in \eqref{eq:slab} using a
semi--infinite programming approach.  Without any loss of generality,
we again assume throughout this section that
$\beta^2\geq\alpha^2$.

It is convenient to set up this problem as the semi--infinite program
\[ \min\left\{-\ln\det(A): Ay+c\in\slab,\quad\forall\,y:\, ||y||=1\right\}, \]
in which we represent the inscribed ellipsoid as $E=c+A(B_n)$ where
$A$ is a symmetric, positive definite matrix with $\vol(E)=\omega_n\det A$.
Lemma~\ref{Lemma:invariance} implies that the optimal ellipsoid has
the form $A=\text{diag}(a,b,\ldots,b)$ and $c=\tau e_1$ for some
$a>0$, $b>0$, and $\tau$ in $\R$.  Writing $y=(u,z)$ in $\R\times\R^{n-1}$,
we can replace the above semi--infinite program with a far simpler one
\begin{equation}
\label{eq:SIP-for-MVIE-of-slab}
\begin{split}
 \min\quad &-\ln a-(n-1)\ln b,\\
 \st\quad  &-au-\tau\leq-\alpha,\\
           &au+\tau\leq\beta,\qquad\qquad\qquad\qquad\qquad\forall\, u\in[-1,1]\\
           &(au+\tau)^2+b^2(1-u^2)\leq1,
\end{split}
\end{equation}
in which the decision variables are $(a,b,\tau)$ and the index set is $[-1,1]$.

We make some useful observations before writing down the optimality conditions for
Problem~\ref{eq:SIP-for-MVIE-of-slab}. Theorem~\ref{Theorem:FJ-Theorem-For-SIP}
implies that the optimality conditions will involve at most three active constraints
with corresponding multipliers positive.  Clearly, the first constraint above can be
active only for $u=-1$ and the second one for $u=1$.

Next, if $\alpha>-1$, we claim that the third constraint is active for
at most one index value $u$.  Otherwise, the quadratic function
\[  g(u):= (au+\tau)^2+b^2(1-u^2)-1 \]
is non--positive on the interval $[-1,1]$ and equals zero for two distint
values of $u$. If the function $g(u)$ is actually a linear function (a=b),
then it is identically zero; otherwise, it is easy to see that $g(u)$ must
equal zero at the
endpoints $-1$ and $1$. In all cases, we have $g(-1)=g(1)=0$, so that
$(-a+\tau)^2=1=(a+\tau)^2$.  This gives $a-\tau=1=a+\tau$, since the other
possibilities give $a=0$ or $a=-1$. But then $a=1$, $\tau=0$, and
$1=a-\tau\leq-\alpha$, where the inequality expresses the feasibility of the
first inequality in Problem~\ref{eq:SIP-for-MVIE-of-slab}. We obtain $\alpha=-1$
and $\beta=1$, a contradiction. The claim is proved.

The following theorem is another major result of this paper.

\begin{theorem}
\label{Theorem:optimal-solution-for-slab-MVIE-problem}
The maximal inscribed ellipsoid $\ie(\slab)$ has the form $E=c+A(B_n)$ where
 $c=\tau e_1$, $\alpha<\tau<\beta$, $A=\diag(a,b,\ldots,b)$ with $a>0$, $b>0$,
satisfying the following conditions:\\
(i) If $\alpha=-\beta$, then
\begin{equation}
\label{eq:solution-i-max}
\tau=0,\qquad a=\beta,\qquad b=1.
\end{equation}
(ii) If $4n(1-\alpha^2)<(n+1)^2(\beta^2-\alpha^2)$, then
\begin{equation}
\label{eq:solution-ii-max}
\begin{split}
    \tau
&=  \frac{\alpha+\sqrt{\alpha^2+4n(1-\alpha^2)/(n+1)^2}}{2},\\
    a
&=  \tau-\alpha,\qquad
b^2
=   a(a+n\tau),
\end{split}
\end{equation}
(iii) If $4n(1-\alpha^2)\geq(n+1)^2(\beta^2-\alpha^2)$ and $\alpha\neq-\beta$, then
\begin{equation}
\label{eq:solution-iii-max}
\begin{split}
       \tau
&=     \frac{\beta+\alpha}{2},\qquad
       a
=      \frac{\beta-\alpha}{2},\\
       b^2
&=     a^2+
      \left(
          \frac{\beta^2-\alpha^2}
          {2(\sqrt{1-\alpha^2}-\sqrt{1-\beta^2})}
      \right)^2.
\end{split}
\end{equation}
\end{theorem}
%%%
\begin{proof}
If $\alpha=-1$, then $\beta=1$ and the optimal ellipsoid is the unit ball $B_n$,
which agrees with (i) of the theorem. We assume in the rest of the proof that $\alpha>-1$.

We saw above that each of the constraints in \ref{eq:SIP-for-MVIE-of-slab} can be active
for at most one value of $u$ in $[-1,1]$, and the first and second constraints for $u=-1$
and $u=1$, respectively.  Then Theorem~\ref{Theorem:FJ-Theorem-For-SIP} gives the
optimality conditions
\begin{equation}
\label{eq:FJ-conditions-for-MVIE-of-slab}
\begin{split}
      \lambda_1+\lambda_2+\delta(au+\tau)u
&=    \frac{\lambda_0}{a},\\
      \delta b(1-u^2)
&=    \frac{(n-1)\lambda_0}{b},\quad u\in[-1,1],\\
      -\lambda_1+\lambda_2+\delta(au+\tau)
&=    0,
\end{split}
\end{equation}
where the non--negative multipliers $(\lambda_0,\lambda_1,\lambda_2,\delta/2)\neq0$
correspond to the objective function, and the first, second, and third (active)
constraints in \ref{eq:SIP-for-MVIE-of-slab}, respectively.

Our \emph{first} claim is that $\lambda_0>0$. Otherwise, the second equation in
\eqref{eq:FJ-conditions-for-MVIE-of-slab} gives $\delta=0$ or $u=\mp1$.
If $\delta=0$, then the first equation gives the contradiction $\lambda_1=\lambda_2=0$.
If $u=-1$, then we have $\mp1=a-\tau\leq-\alpha<1$, implying $a-\tau=-1$ or $\tau=a+1>1$,
another contradiction. If $u=1$, we have $\mp1=a+\tau$. If $a+\tau=1$, then
the first equation in \eqref{eq:FJ-conditions-for-MVIE-of-slab} gives $\lambda_1+\lambda_2+\delta=0$,
that is, $\lambda_1=\lambda_2=\delta=0$, a contradiction.  If $a+\tau=-1$, then $\tau=-a-1<-1$,
yet another contradiction.  The claim is proved.  We set $\lambda_0=1$.

The second equation in \eqref{eq:FJ-conditions-for-MVIE-of-slab} gives $\delta>0$,
and that $g(u)=0$ for some $u$ in $(-1,1)$ and negative elsewhere on $[-1,1]$.
Note this means that $u$ is the global maximum as well as the unique root of $g$ on
$\R$, leading to the conditions
\begin{equation}
\label{eq:for-g-in-FJ-conditions-for-MVIE-of-slab}
     b>a,\quad
     u
=    \frac{a\tau}{b^2-a^2},\quad
     b^2\tau^2
=    (1-b^2)(b^2-a^2),
\end{equation}
where the last equation expresses the fact that the discriminant of
$g$ equals zero. We will also have occasion to use the equality
\begin{equation}
\label{eq:au+tau=bsq-tau-over-(bsq-asq)}
     au+\tau
=    a\frac{a\tau}{b^2-a^2}+\tau
=    \frac{b^2\tau}{b^2-a^2}.
\end{equation}

Our \emph{second} claim is  that
\begin{equation}
\label{eq:-a+tau=alpha}
-a+\tau=\alpha.
\end{equation}
If not, then $-a+\tau>\alpha$, $\lambda_1=0$, and the third equation in
\eqref{eq:FJ-conditions-for-MVIE-of-slab} gives $\lambda_2=-\delta(au+\tau)$.
Substituting this into the first equation in
\eqref{eq:FJ-conditions-for-MVIE-of-slab} leads to $\delta(au+\tau)(u-1)=a^{-1}$,
and consequently to $au+\tau<0$.  But then $\lambda_2>0$ and $a+\tau=\beta$,
which together with $-a+\tau>\alpha$ gives $\tau>(\alpha+\beta)/2\geq0$,
that is, $\tau>0$.  Moreover, the second equation in
\eqref{eq:for-g-in-FJ-conditions-for-MVIE-of-slab} gives $u>0$,
and this leads to the contradiction that $-\lambda_2=\delta(au+\tau)>0$.
The claim is proved.

Next, we have
\begin{equation*}
\begin{split}
       \frac{1}{a}-2\lambda_2
&=
       \delta(au+\tau)(u+1)
=      \frac{n-1}{b^2(1-u^2)}\cdot\frac{b^2\tau}{b^2-a^2}(u+1)\\
&=     \frac{(n-1)\tau}{(1-u)(b^2-a^2)}
=      \frac{(n-1)u}{a(1-u)},
\end{split}
\end{equation*}
where the first equality is obtained by adding the first and third equations in
\eqref{eq:FJ-conditions-for-MVIE-of-slab}, the second equality follows by substituting
the value of $\delta$ from the second equation in \eqref{eq:FJ-conditions-for-MVIE-of-slab}
and the value of $au+\tau$ from \eqref{eq:au+tau=bsq-tau-over-(bsq-asq)}, and the
last equality follows by substituting the value $\tau/(b^2-a^2)=u/a$ from the first
equation in \eqref{eq:for-g-in-FJ-conditions-for-MVIE-of-slab}.  Therefore,
\[  \lambda_2=\frac{1-nu}{2a(1-u)}.  \]
Consequently, $u\leq1/n$ and $u=1/n$ if and only if $\lambda_2=0$.
Furthermore, we have $u\geq0$: if $u<0$, then $\tau<0$ by virtue of the first
equation in \eqref{eq:for-g-in-FJ-conditions-for-MVIE-of-slab}, so that $au+\tau<0$.
But then the third equation in \eqref{eq:FJ-conditions-for-MVIE-of-slab} gives
$\lambda_2>\lambda_1\geq0$, implying $a+\tau=\beta$. This and \eqref{eq:-a+tau=alpha}
give $\tau=(\alpha+\beta)\geq0$, a contradiction.  Therefore,
\begin{equation}
\label{eq:0<=u<=1/n}
0\leq u\leq\frac{1}{n},\quad\text{and}\quad
u=\frac{1}{n}\quad\Longleftrightarrow\quad\lambda_2=0.
\end{equation}

We can now prove part (i) of the theorem.  We first show that
\[ u=0\quad\Longleftrightarrow\quad \alpha=-\beta. \]
If $\alpha=-\beta$, then Lemma~\ref{Lemma:invariance} implies that $\tau=0$, which in
turn implies $u=0$.  Conversely, if $u=0$, then $\tau=0$ and $\lambda_2>0$, and we have
$-a+\tau=\alpha$ and $a+\tau=\beta$, leading to $0=\tau=(\beta+\alpha)/2$.
Consequently, \eqref{eq:-a+tau=alpha} gives $a=-\alpha=\beta$ and the second equation in
\eqref{eq:for-g-in-FJ-conditions-for-MVIE-of-slab} gives $b=1$.

We now consider the remaining cases $0<u\leq1/n$.  We note that
\begin{equation*}
     (au+\tau)u\tau
=    \frac{u(b^2\tau^2)}{b^2-a^2}
=    u(1-b^2)
=    u(1-a^2)-a\tau,
\end{equation*}
where the first equality follows from \eqref{eq:au+tau=bsq-tau-over-(bsq-asq)}
and last two equalities from \eqref{eq:for-g-in-FJ-conditions-for-MVIE-of-slab},
leading to a quadratic equation for $u$,
\begin{equation}
\label{eq:quadratic-equation-for-u}
(a\tau)u^2-(1-a^2-\tau^2)u+a\tau=0.
\end{equation}
Define $\varepsilon\geq0$ such that $a+\tau=\beta-\varepsilon=:\beta_\varepsilon$.
The equation $a+\tau=\beta_\varepsilon$ together with equation $-a+\tau=\alpha$
from \eqref{eq:-a+tau=alpha} give
\begin{equation}
\label{eq:values-of-a-and-tau}
     \tau
=     \frac{\beta_\varepsilon+\alpha}{2}>0,
      \quad
      a
=     \frac{\beta_\varepsilon-\alpha}{2}>0.
\end{equation}
Substituting these in \eqref{eq:quadratic-equation-for-u} gives another
quadratic equality for $u$,
\begin{equation}
\label{eq:for-u-in-MVIE-2}
(\beta_\varepsilon^2-\alpha^2)u^2
  -2(2-\beta_\varepsilon^2-\alpha^2)u
   +(\beta_\varepsilon^2-\alpha^2)=0.
\end{equation}
It is easy to verify, using \eqref{eq:0<=u<=1/n}, that
\begin{equation}
\label{eq:implications-for-lambda2}
\begin{split}
          \lambda_2>0
&\quad\Longleftrightarrow\quad
          0<u<\frac{1}{n}
\quad\Longleftrightarrow\quad
          (n+1)^2(\beta_\varepsilon^2-\alpha^2)<4n(1-\alpha^2),\\
          \lambda_2=0
&\quad\Longleftrightarrow\quad
          u=\frac{1}{n}
\qquad\quad\Longleftrightarrow\quad
          (n+1)^2(\beta_\varepsilon^2-\alpha^2)=4n(1-\alpha^2).\\
\end{split}
\end{equation}
Here the second equivalence on the first line follows because the quadratic equation
for $u$ in \eqref{eq:for-u-in-MVIE-2} has negative value at $u=1/n$. Since the
leading term of the quadratic function is positive, its two roots $r_1<r_2$ are positive,
their product is 1, $0<r_1<1/n$, and $r_2>n$.

We now make our \emph{third} and important claim that
\begin{equation}
\label{eq:a-plus-tau<beta-iff}
      a+\tau<\beta\quad\text{iff}\quad
      4n(1-\alpha^2)<(n+1)(\beta^2-\alpha^2).
\end{equation}
On the one hand, if $a+\tau=\beta$, then $\varepsilon=0$ and \eqref{eq:implications-for-lambda2}
gives $(n+1)(\beta^2-\alpha^2)<4n(1-\alpha^2)$ or $(n+1)(\beta^2-\alpha^2)=4n(1-\alpha^2)$,
depending on whether $\lambda_2>0$ or $\lambda_2=0$, respectively. In either case,
we have $(n+1)(\beta^2-\alpha^2)\leq 4n(1-\alpha^2)$.
On the other hand, if $a+\tau<\beta$, then $\lambda_2=0$, $\varepsilon>0$, and
\eqref{eq:implications-for-lambda2} gives
$(n+1)(\beta_\varepsilon^2-\alpha^2)=4n(1-\alpha^2)$, that is,
$4n(1-\alpha^2)<(n+1)(\beta^2-\alpha^2)$.  The claim is proved.

The computation of the decision variables $(a,b,\tau)$ in the cases (ii) and (iii) now becomes
a routine matter.  If $4n(1-\alpha^2)<(n+1)^2(\beta^2-\alpha^2)$,
then \eqref{eq:a-plus-tau<beta-iff} implies $a+\tau<\beta$, and we have $\lambda_2=0$,
$u=1/n$. Substituting the value $a=\tau-\alpha$ from \eqref{eq:-a+tau=alpha} into
\eqref{eq:quadratic-equation-for-u} gives the quadratic equation for $\tau$,
\[  (n+1)^2\tau^2-(n+1)^2\alpha\tau-n(1-\alpha^2)=0. \]
Since $\tau>0$, the feasible root is given by
\begin{equation*}
    \tau
=   (\alpha+\sqrt{\alpha^2+4n(1-\alpha^2)/(n+1)^2})/2.
\end{equation*}
Then \eqref{eq:-a+tau=alpha} gives $a=\tau-\alpha$ and
\eqref{eq:for-g-in-FJ-conditions-for-MVIE-of-slab} gives
$1/n=u=(a\tau)/(b^2-a^2)$, that is, $b^2=a^2+na\tau$.

If $4n(1-\alpha^2)\geq(n+1)^2(\beta^2-\alpha^2)$,
then \eqref{eq:a-plus-tau<beta-iff} and \eqref{eq:for-u-in-MVIE-2} give
\[
      u
=     \frac{\left(\sqrt{1-\alpha^2}\mp\sqrt{1-\beta^2}\right)^2}
      {\beta^2-\alpha^2}.
\]
It is easy to verify that the condition $u<1$ is equivalent to
$\sqrt{1-\beta^2}\mp\sqrt{1-\alpha^2}<0$, which is impossible if we choose the
plus sign.  Thus the feasible root is the one with the negative sign.
Finally, the first equation in \eqref{eq:for-g-in-FJ-conditions-for-MVIE-of-slab}
gives $b^2=a^2+(a\tau)/u$, or more explicitly the formula for $b^2$ in
\eqref{eq:solution-iii-max}.
\end{proof}

We end this section by reducing the semi--infinite program
\eqref{eq:SIP-for-MVIE-of-slab} to an ordinary
nonlinear programming problem.  However, we do not attempt to solve the
resulting program in order to save space. As we already noted, the linear
constraints, the first two inequalities in problem~\eqref{eq:SIP-for-MVIE-of-slab},
simply reduce to the constraints $a-\tau\leq-\alpha$ and $a+\tau\leq\beta$.
In order to reduce the quadratic inequality system to a set of ordinary
inequalities, we use a theorem of Luk\'acs characterizing the class of
non--negative polynomials on a given interval.  A simple inductive proof of
Luk\'acs's Theorem can be found in \cite{BrickmanSteinberg62}.  For a quadratic
polynomial $q(u)$ on the interval $I=[a,b]$, this theorem states that $q$ is
non--negative on $I$ if and only if there exist scalars $\alpha$, $\beta$,
and $\gamma\geq0$ such that
\begin{equation}
\label{eq:Lukacs-theorem-for-a-quadratic}
  q(u)=(\alpha u+\beta)^2+\gamma(u-a)(b-u).
\end{equation}
Since the proof is short and simple in this case, we give it, following
\cite{BrickmanSteinberg62}. Note that the polynomial $p(u):=q(u)-l(u)^2$,
where $l(u)=[\sqrt{q(a)}(u-b)+\sqrt{q(b)}(u-a)]/(b-a)$, satisfies $p(a)=0=p(b)$,
so that there exists a constant $\gamma$ such that $p(u)=\gamma(u-a)(b-u)$.
Clearly, \eqref{eq:Lukacs-theorem-for-a-quadratic} holds true if we can show that
$\gamma\geq0$.
Note that $l(a)\leq0$ and $l(b)\geq0$ so that $l(u)=d(u-r)$ for some constant $d$
and $r\in[a,b]$.  We have $0\leq q(u)=d^2(u-r)^2+\gamma(u-a)(b-u)$, or
\[ -d^2(u-r)^2 \leq \gamma(u-a)(b-u),\quad \forall\, u\in[a,b]. \]
If $r$ is in $(a,b)$, then choosing $u=r$ gives $\gamma\geq0$.  If $r=a$ or $r=b$, then
choosing $u$ near $r$ gives $\gamma\geq0$.  This completes the proof.

We remark that the result we just proved also follows from the one dimensional case
of the S--procedure, see \cite{Polyak99} or \cite{PolikTerlaky07}.

Applying this result to our quadratic function $-q(u)=-(au+\tau)^2-b^2(1-u^2)+1$ which
is non--negative on the interval $[-1,1]$, we see that there exists scalars $c,\,d$, and
$\gamma\geq0$ such that
\[ (au+\tau)^2+b^2(1-u^2)-1 = -(cu+d)^2+\gamma(u^2-1), \]
that is, $a^2-b^2+c^2-\gamma=0$, $a\tau+cd=0$, and $b^2+\tau^2+d^2+\gamma=1$.
Therefore, the problem of finding $\ie(B_{\alpha\beta})$ reduces to the nonlinear
programming problem
\begin{equation*}
\label{eq:NLP-for-MVIE-of-slab}
\begin{split}
 \min\quad &-\ln a-(n-1)\ln b,\\
 \st\quad  &-a+\tau\geq\alpha,\\
           &a+\tau\leq\beta,\\
           &a^2-b^2+c^2-\gamma=0,\\
           &a\tau+cd=0,\\
           &b^2+\tau^2+d^2+\gamma=1,\\
           &\gamma\geq0,
\end{split}
\end{equation*}
in which the decision variables are $(a,b,\tau,c,d,\gamma)$ and $\alpha,\,\beta$ are
parameters.

%*************************************************************************************************

\section{Determination of the minimum volume covering ellipsoid of a truncated second
order cone or a cylinder}
\label{sec:MVCE-of-convex-combination}

In this section, one of the problems we are interested in is finding the minimum volume
ellipsoid covering the truncated second order cone
\[  K = \{x=(x_1,\bar{x})\in\R\times\R^{n-1}: ||B(\bar{x}-c)||\leq x_1,\  a\leq x_1\leq b\}, \]
where $B$ is an invertible matrix in $\R^{(n-1)\times(n-1)}$ and $0\leq a<b$ are constants.  By an affine
change of $\bar{x}$, we may assume that $c=0$ and $B=I_{n-1}$.  We claim that by further
affine change of variables, we can reduce the convex body $K$ to have the form
\[  \yayik:=\co(S_\alpha\cup S_\beta), \]
where $S_\alpha:=\{x\in B_n: x_1=\alpha\}$, $S_\beta:=\{x\in B_n: x_1=\beta\}$, and
$-1\leq\alpha<\beta\leq1$. Consider the ball $B$ in $\R^n$ with center $(a+b)e_1$ and radius
$\sqrt{a^2+b^2}$. The slice $P_a:=\{(a,\bar{x})\in\R^n: ||\bar{x}||\leq a\}\subset K$ lies in
$B$, since $||(a,\bar{x})-(a+b,0)||^2=b^2+||\bar{x}||^2\leq a^2+b^2$, and similarly $P_b\subset B$.
A further translation and then scaling transforms $B$ into $B_n$. This proves the claim.
Conversely, if $-\alpha\neq\beta$, $\yayik$ can be viewed as a truncated second order cone.

The other problem we are interested in is finding the minimum volume ellipsoid covering a
cylinder. If $-\alpha=\beta$, then $\yayik$ is clearly a cylinder. Conversely, any cylinder
can be transformed into such a $\yayik$ by an affine transformation.  Consequently, the
$\ce(\yayik)$ problem formulates both problems at the same time.

\begin{theorem}
\label{Theorem:optimal-solution-for-covex-combi-slice-MVCE-problem}
The ellipsoid $\ce(\yayik)$ has the form $E(X,c)$ where $c=\tau e_1$ and $X=\diag(a,b,\ldots,b)$,
where the parameters $a>0$, $b>0$, and $\alpha<\tau<\beta$ are given as follows:\\
(i) If $\alpha\beta\,=\,-1/n$, then
\begin{equation}
    \label{eq:Solution-i-slice}
    \tau=0,\,\text{ and }\,a=b=1.
\end{equation}
(ii) If $\alpha+\beta\,=\,0$ and $\alpha\beta\,\not=\,-1/n$, then
\begin{equation}
    \label{eq:Solution-ii-slice}
    \tau=0,\quad
    a=\frac{1}{n\beta^2},\quad
    b=\frac{n-1}{n(1-\beta^2)}.
\end{equation}
(iii) If $\alpha+\beta\,\neq0$ and $\alpha\beta\,\not=\,-1/n$, then
\begin{equation}
\label{eq:Solution-iii-slice}
\begin{split}
    \tau
&=  \frac{n(\beta+\alpha)^2+2(1+\alpha\beta)-
    \sqrt{\Delta}}{2(n+1)(\beta+\alpha)},\\
    a
&=  \frac{1}{n(\tau-\alpha)(\beta-\tau)},\quad
    b
=   \frac{1-a(\tau-\alpha)^2}{1-\alpha^2},
\end{split}
\end{equation}
where $\Delta=n^2(\beta^2-\alpha^2)^2+4(1-\alpha^2)(1-\beta^2)$.
\end{theorem}
\begin{proof}
It is clear that $\Aut(\yayik)\supseteq\Aut(\slab)$, so that Lemma \ref{Lemma:invariance} implies
$X=diag(a,b,...,b)$ and $c=\tau e_1$ with $a>0$, $b>0$ and $\tau\in\R$.
Let $\{u_i\}_{i=1}^k$ be the contact points of the optimal ellipsoid with $\yayik$.
Writing $u_i=(y_i,z_i)\in\R\times\R^{n-1}$, we have that $y_i$ is either $\alpha$ or $\beta$, and
$1=||u_i||^2=y_i^2+||z_i||^2$, that is, $||z_i||^2=1-y_i^2$.  Since
$  (u_i-c)(u_i-c)^T
=  \begin{bmatrix}
      (y_i-\tau)^2  &(y_i-\tau)z_i^T\\
      (y_i-\tau)z_i &z_iz_i^T
   \end{bmatrix}$,
Theorem \ref{Theorem:FJ-conditions-for-MVCE-SIP} yields the following optimality conditions:
\begin{eqnarray}
           \tau
&=&\
           \frac{1}{n}\sum_{i=1}^k\lambda_iy_i,\quad
           0
=
           \sum_{i=1}^k\lambda_iz_i,\quad
           \sum_{i=1}^k\lambda_i
=          n,\label{eq:FJ-for-slice-MVCE-1} \\
           \frac{1}{a}
&=&\
           \sum_{i=1}^k\lambda_i(y_i-\tau)^2,\quad
           0
=
           \sum_{i=1}^k\lambda_i(y_i-\tau)z_i,\quad
           \frac{1}{b}I_{n-1}
=          \sum_{i=1}^k\lambda_iz_iz_i^T,
           \label{eq:FJ-for-slice-MVCE-2}\\
           0
&=&\
           a(y_i-\tau)^2+b(1-y_i^2)-1,\quad i=1,\ldots,k,
           \label{eq:FJ-for-slice-MVCE-3}\\
           0
&\geq&\    a(y-\tau)^2+b(1-y^2)-1,\quad \text{for } y\in\{\alpha,\beta\}.
           \label{eq:FJ-for-slice-MVCE-4}
\end{eqnarray}
Here the last line expresses the feasibility condition $\yayik\subseteq E(X,c)$,
since $\yayik\subseteq E(X,c)$ if and only if
$\partial S_\alpha\cup\partial S_\beta\subseteq E(X,c)$.

The conditions \eqref{eq:FJ-for-slice-MVCE-1}--\eqref{eq:FJ-for-slice-MVCE-4}
characterize the ellipsoid $\ce(\yayik)$. Since the ellipsoid is unique, its parameters
$(a,b,\tau)$ are unique and can be recovered from the above conditions.

The same arguments in Lemma~\ref{Lemma:a-neq-b} applies here and gives the
equation \eqref{eq:a-b-tau},
\begin{equation}
\label{eq:a-b-tau-yayik}
    1
=  a(\alpha-\tau)^2+b(1-\alpha^2),\quad
    \tau
=   \left(1-\frac{b}{a}\right)\cdot\frac{\alpha+\beta}{2},\quad
    \frac{1}{a}
=  n(\tau-\alpha)(\beta-\tau),
\end{equation}
which we again use to compute the variables $a$, $b$, and $\tau$.
If $a=b$ in the optimal ellipsoid, \eqref{eq:a-b-tau-yayik}
immediately gives $\tau=0$, $a=b=1$, and $\alpha\beta=-1/n$.

Next, if $a\neq b$ and $\alpha=-\beta$, then \eqref{eq:a-b-tau-yayik} gives $\tau=0$,
$a=1/(n\alpha^2)$ and $b=(n-1)/(n(1-\alpha^2))$. Furthermore, if we have $a=1=n\alpha^2$,
then we obtain a contradiction since $b=(n-1)/(n-1)=1=a$. Thererefore, $a\neq1$ and the
last equation in \eqref{eq:a-b-tau-yayik} gives $\alpha\beta\neq-1/n$.

Lastly, $a\neq b$ and $\alpha\neq-\beta$, then the same argument in
Lemma~\ref{Lemma:a-neq-b} gives the quadratic equation
\eqref{eq:quadratic-equation-for-tau} for $\tau$,
\[
    (n+1)(\alpha+\beta)\tau^2
     -\left(n(\alpha+\beta)^2+2(1+\alpha\beta)\right)\tau
      +(\alpha+\beta)(1+n\alpha\beta)=0,
\]
and the resulting equations in \eqref{eq:Solution-iii-slice} for $(a,b,\tau)$.
Since the middle equation in \eqref{eq:a-b-tau-yayik} implies $\tau\neq0$, we have
\[
    0
\neq
    (n(\alpha+\beta)^2+2(1+n\alpha\beta))^2-\Delta
=
    4(n+1)(\alpha+\beta)^2(1+n\alpha\beta),
\]
that is, $\alpha\beta\neq-1/n$.
\end{proof}

%*************************************************************************************************
\bibliography{Ellipsoid}
%*************************************************************************************************

\end{document}